\documentclass[12pt]{amsart}
\usepackage[utf8]{inputenc}
\usepackage[T1]{fontenc}
\usepackage[frenchb,english]{babel}

\usepackage{amsmath}
\usepackage{amssymb}
\usepackage{amsfonts}
\usepackage{amsthm}
\usepackage{enumerate}
\usepackage{vmargin}
\usepackage{mathrsfs}
\usepackage{mathtools}
\usepackage{dsfont}
\usepackage{lmodern}
\usepackage{hyperref}

\newcommand{\Z}{\ensuremath{\mathbb{Z}}}
\newcommand{\R}{\ensuremath{\mathbb{R}}}
\newcommand{\C}{\ensuremath{\mathbb{C}}}
\newcommand{\Cc}{\ensuremath{\mathcal{C}}}
\newcommand{\D}{\ensuremath{\mathbb{D}}}

\newcommand{\Id}{{\rm Id}}

\renewcommand{\le}{\ensuremath{\leqslant}}
\renewcommand{\ge}{\ensuremath{\geqslant}}
\renewcommand{\leq}{\ensuremath{\leqslant}}
\renewcommand{\geq}{\ensuremath{\geqslant}}

\newcommand{\U}{\mathcal{U}}

\newcommand\be{\begin{eqnarray}}
\newcommand\ee{\end{eqnarray}}
\newcommand{\norm}[1]{ \| #1  \|}
\newcommand{\bnorm}[1]{ \big\| #1  \big\|}

\newcommand{\bgnorm}[1]{ \bigg\| #1  \bigg\|}
\newcommand{\Bgnorm}[1]{ \Bigg\| #1  \Bigg\|}

\newcommand{\ot}{\otimes}
\newcommand{\co}{\colon}

\def\Ker{{\rm Ker} \, }

\def\Re{{\rm Re} }

\newcommand{\ovl}{\overline}

\newcommand{\ul}{\mathcal{U}}
\newcommand{\epsi}{\varepsilon}
\newcommand{\Rad}{{\rm Rad}}

\newcommand{\Ran}{{\rm Ran}}

\newcommand{\biggnorm}[1]{\biggl\lVert#1\biggr\rVert}

\newtheorem{thm}{Theorem}[section]
\newtheorem{prop}[thm]{Proposition}
\newtheorem{cor}[thm]{Corollary}
\newtheorem{lemma}[thm]{Lemma}

\theoremstyle{definition}

\newtheorem{remark}[thm]{Remark}

\newcommand{\beq}{\begin{equation}}
\newcommand{\eeq}{\end{equation}}

\numberwithin{equation}{section}

\begin{document}
\selectlanguage{english}
\title[Isometric dilations and $H^\infty$ calculus]{Isometric dilations and $H^\infty$ calculus for bounded analytic
semigroups and Ritt operators}

\author{C\'edric Arhancet}
\address{Laboratoire de Math\'ematiques, Universit\'e de Franche-Comt\'e,
25030 Besan\c{c}on Cedex,  France}
\email{cedric.arhancet@univ-fcomte.fr}

\author{Stephan Fackler}
\address{Institute of Applied Analysis, University of Ulm, Helmholtzstr. 18, 89069 Ulm, Germany}
\email{stephan.fackler@uni-ulm.de}

\author{Christian Le Merdy}
\address{Laboratoire de Math\'ematiques, Universit\'e de Franche-Comt\'e,
25030 Besan\c{c}on Cedex,  France}
\email{christian.lemerdy@univ-fcomte.fr}


\subjclass[2010]{Primary 47A60; Secondary 47D06, 47A20, 22D12}
\keywords{Dilation, Ritt operator, sectorial operator, group representation, 
functional calculus, semigroup, amenable group}

\begin{abstract}
We show that any bounded analytic semigroup on $L^p$ (with $1<p<\infty$)
whose negative generator admits a bounded $H^{\infty}(\Sigma_\theta)$ 
functional calculus for some $\theta \in (0,\frac{\pi}{2})$
can be dilated into a bounded analytic semigroup $(R_t)_{t\geq 0}$
on a bigger $L^p$-space in such a way that $R_t$ is a positive
contraction for any $t\geq 0$. We also establish a discrete analogue 
for Ritt operators and consider the case when $L^p$-spaces
are replaced by more general Banach spaces. In connection
with these functional calculus issues, we study isometric 
dilations of bounded continuous representations of amenable groups
on Banach spaces and establish various generalizations of 
Dixmier's unitarization theorem.
\end{abstract}

\maketitle


\section{Introduction}


In \cite[Remark 4.c]{Weis01}, Weis showed that if $(T_t)_{t \geq 0}$ is a
bounded analytic semigroup on a space $L^p(\Omega)$ (with $1<p<\infty$) 
such that each operator $T_t \co L^p(\Omega) \to L^p(\Omega)$ is a positive contraction,
then its negative generator $A$ admits a bounded $H^{\infty}(\Sigma_\theta)$ functional calculus 
for some angle $0<\theta < \frac{\pi}{2}$. Here and later on in this paper,
\begin{equation}\label{sector}
\Sigma_\theta=\{z\in\C^*\ :\ \vert{\rm Arg}(z)\vert <\theta\}
\end{equation} 
denotes the open sector of angle $2\theta$ around the positive real axis $\R_+^*$. 
The first main result of this paper is the following converse, which says that this class of semigroups
generates {\it by dilation} the class of all bounded analytic semigroups on an $L^p$-space
whose negative generator admits a bounded $H^{\infty}(\Sigma_\theta)$ functional calculus 
for some $\theta \in (0,\frac{\pi}{2})$. More precisely, we will prove the following.

\begin{thm}
\label{Th main sectorial Lp 0}
Let $\Omega$ be a measure space and let $1<p<\infty$. Let $(T_t)_{t \geq 0}$ be a 
bounded analytic semigroup on $L^p(\Omega)$ and assume that its negative generator 
admits a bounded $H^\infty(\Sigma_{\theta})$ functional calculus for some 
$0<\theta < \frac{\pi}{2}$. Then there exist a measure space 
$\Omega'$, a bounded analytic semigroup $(R_t)_{t \geq 0}$ on the space $L^p(\Omega')$
such that each $R_t \co L^p(\Omega') \to L^p(\Omega')$ is a positive contraction,
and two bounded operators $J \co L^p(\Omega) \to L^p(\Omega')$ and $Q \co L^p(\Omega') \to L^p(\Omega)$ 
such that 
$$
T_{t}=QR_t J,\qquad \text{for all }t \geq 0.
$$
\end{thm}

Note that in the above situation, $J$ is an isomorphic embedding 
whereas $JQ$ is a bounded projection. Hence the new space $L^p(\Omega')$ can be seen as
a bigger space than the initial $L^p(\Omega)$, containing the latter as a 
complemented subspace.

Theorem \ref{Th main sectorial Lp 0} improves a recent result by the second named 
author on the structure of $L^p$-semigroups with a bounded $H^\infty$ functional calculus \cite{Fac13c}. 
Together with Weis's theorem, this  provides a complete characterization 
of bounded $H^\infty$ functional calculus on $L^p$-spaces. This characterization should be regarded
as an $L^p$-analogue of the theorem from \cite{Mer98} which says that if $(T_t)_{t\geq 0}$
is  a bounded analytic semigroup on some Hilbert space, with negative generator $A$,
then $A$ admits a bounded $H^{\infty}(\Sigma_\theta)$ functional calculus 
for some angle $0<\theta < \frac{\pi}{2}$ if and only if $(T_t)_{t\geq 0}$ is similar
to a contractive semigroup.

$H^{\infty}$ functional calculus is a very useful and important tool in 
various areas: harmonic analysis of semigroups, multiplier theory, Kato's square root problem,  
maximal regularity in parabolic equations, control theory, etc.
It grew up from the two fundamental papers \cite{MC86,CDMY96}. For detailed information
we refer the reader to \cite{Haa06}, \cite{KuW04}, to the survey papers 
\cite{AArendt-04-Survey}, \cite{LeMerdy-SurveymMaximalRegularity}, 
\cite{LeMerdy-SurveySquareFunctions}
and \cite{Weis-06-Survey}, and to the references therein. 

We will also establish an analogue of Theorem \ref{Th main sectorial Lp 0} for Ritt operators.
This class of operators can be regarded as the discrete analogue 
of the class of bounded analytic semigroups. Ritt operators have 
a natural and fruitful notion of $H^\infty$ functional calculus 
with respect to Stolz domains $B_\gamma$, 
see Section~\ref{sec:background} below and \cite{Mer12}. We will
show that a Ritt operator $T\colon L^p(\Omega)\to L^p(\Omega)$ has a bounded 
$H^{\infty}(B_\gamma)$ functional calculus 
for some $0<\gamma < \frac{\pi}{2}$ (if and) only if there exist a measure space 
$\Omega'$, a contractive and positive Ritt operator $R\colon L^p(\Omega')\to L^p(\Omega')$,
and two bounded operators $J \co L^p(\Omega) \to L^p(\Omega')$ and $Q \co L^p(\Omega') \to L^p(\Omega)$ 
such that $T^n=Q R^n J$ for all integer $n\geq 0$.

Theorem \ref{Th main sectorial Lp 0} and its discrete version above are
proved in Section~\ref{sec:main_results}. We also give analogous results 
when $L^p$-spaces
are replaced by other classes of Banach spaces, such as  UMD spaces, noncommutative
$L^p$-spaces, or quotients of closed subspaces of $L^p$-spaces.

Section~\ref{sec:dilations}
is devoted to closely related but different dilation results. 
Let $(T_t)_{t\geq 0}$ be a bounded analytic semigroup on a 
Banach space $X$, with negative generator $A$, and let $1<p<\infty$. 
Under mild conditions on $X$ we show that if  
$A$ admits a bounded $H^{\infty}(\Sigma_\theta)$ functional calculus 
for some angle $0<\theta < \frac{\pi}{2}$, then there exist a measure space
$\Omega'$, a $C_0$-group $(U_t)_{t\in\R}$ of isometries on $L^p(\Omega';X)$
and two bounded operators $J \co X \to L^p(\Omega';X)$ and $Q \co L^p(\Omega';X) \to X$ 
such that $T_t = QU_tJ$ for all $t\geq 0$. This result improves a well-known dilation 
result of Fr\"ohlich-Weis for operators with a bounded $H^\infty$ functional
calculus \cite{FroWei06}. In a slightly different framework, the Fr\"ohlich-Weis theorem yields 
a dilation $T_t = QU_t J$ where $(U_t)_{t\in\R}$
is a {\it bounded} $C_0$-group. Passing from a bounded group to an isometric
one makes a crucial difference for the applications that we 
develop in Section~\ref{sec:main_results}.

This led us to the question whether a bounded $C_0$-group can be dilated 
into an isometric one. Section~\ref{sec:representations} 
is devoted to this issue, in the more general framework
of amenable group representations. In the case of $L^p$-spaces, we show the following result.

\begin{thm}\label{DD}
Let $G$ be an amenable locally compact group, let $\Omega$ be a measure 
space, let $1<p<\infty$ and
let $\pi \co G \to \mathcal{B}(L^p(\Omega))$ be a bounded strongly 
continuous representation.
Then there exist a measure space $\Omega'$, a strongly
continuous isometric representation $\pi' \co G \to \mathcal{B}(L^p(\Omega'))$,
and two bounded 
operators $J \co L^p(\Omega) \to L^p(\Omega')$ and 
$Q \co L^p(\Omega') \to L^p(\Omega)$ such that 
$$
\pi(t) = Q\pi'(t)J, \qquad \text{for all } t \in G.
$$
\end{thm}

We also show versions of this theorem when $L^p(\Omega)$ is replaced by a more general 
Banach space, and similarity results are established. Indeed Theorem \ref{DD}
can be regarded as an $L^p$-analogue of  Dixmier's unitarization theorem
which says that any bounded strongly continuous representation of an amenable
group on some Hilbert space is similar to a unitary one.

Section~\ref{sec:background} is a preliminary one. It contains an introduction
to sectorial operators, Ritt operators 
and their associated holomorphic functional calculus, 
and provides some background on compressions
and on ordered spaces. In Section~\ref{sec:fractional_powers}, 
we investigate `fractional powers' of power bounded operators, following \cite{Dun11}. These
operators are crucial for the Ritt versions of our results.

\bigskip
We conclude this introduction with a few conventions to be used in this paper. Unless stated
otherwise, the Banach spaces we consider are complex.
Given any Banach spaces $X,Y$, we let $\mathcal{B}(X,Y)$ denote the space of all bounded
operators from $X$ into $Y$. We denote this space by
$\mathcal{B}(X)$ when $Y=X$. We write $\Id_X$ for the identity 
operator on $X$, or simply 
$\Id$ if there is no ambiguity on $X$.

Let $1<p<\infty$. A Banach space $X$ is called an $SQ_p$-space 
(for subspace of quotient of $L^p$) provided that
there exist a measure space $\Omega$ and two closed subspaces
$F\subset E\subset L^p(\Omega)$ such that $X$ is isometrically 
isomorphic to the quotient space $E/F$.

For any non empty open set $\Sigma\subset \C$, we let $H^\infty(\Sigma)$ denote the algebra of all bounded 
analytic functions $\varphi\colon\Sigma\to \C$, equipped with the supremum norm
$$
\norm{\varphi}_{H^\infty(\Sigma)} = \sup\bigl\{
\vert\varphi(z)\vert\, :\, z\in\Sigma\bigr\}.
$$

For any $a\in\C$ and any $r>0$, we let $D(a,r)\subset \C$ denote the open disc 
with center $a$ and radius $r$. We simply denote by $\D=D(0,1)$ the open
unit disc centered at $0$. 

Finally given any set $\Omega$ and a subset
$\Lambda\subset\Omega$, we let $\chi_\Lambda\colon \Omega\to \{0,1\}$
denote the characteristic function of $\Lambda$.


\section{Preliminaries}\label{sec:background}


We start this section with some classical definitions and results
on sectorial operators 
and their associated functional calculus. The construction and basic 
properties below go back to \cite{MC86} and \cite{CDMY96}, see also \cite{KW01}, 
\cite{KuW04} and \cite{Haa06} for complements. We refer to \cite{Gold} or \cite{Pazy} for 
some background on $C_0$-semigroups and the subclass of bounded analytic semigroups.

Let $X$ be a Banach space. Let $A\colon D(A) \to X$ be a closed linear operator
with dense domain $D(A) \subset X$ and let $\sigma(A)$ denote its spectrum. 
Recall the definition 
(\ref{sector}).
We say that $A$ is a sectorial operator of type 
$\mu \in (0,\pi)$ if $\sigma(A) \subset \overline{\Sigma_\mu}$ 
and for any $\nu \in(\mu,\pi)$, the set 
\begin{equation}
\label{2Sectorial}
\Big\{zR(z,A)\, :\, z\in\C \setminus \overline{\Sigma_\nu}\Big\}
\end{equation}
is bounded in $\mathcal{B}(X)$, with $R(z,A)=(z\Id-A)^{-1}$ denoting the resolvent operator. 
It is well-known that an operator $A$ is a sectorial operator of type $<\frac{\pi}{2}$ 
if and only if $-A$ generates a bounded analytic semigroup. 
In this case, this semigroup is denoted by $(e^{-tA})_{t\geq 0}$.

For any $\theta \in (0,\pi)$, let $H^\infty_0(\Sigma_\theta)$ denote the algebra of all bounded 
holomorphic functions $\varphi\co \Sigma_\theta \to \C$ for which there exist 
two positive real numbers $s,K>0$ such that
$$
\vert \varphi(z)\vert\leq\, K\,\frac{\vert z\vert^s}{1+\vert z\vert^{2s}}\,,
\qquad \text{for all } z \in \Sigma_\theta.
$$

Let $0<\mu<\theta<\pi$ and let $\varphi \in H^\infty_0(\Sigma_\theta)$. 
Then for any $\nu \in(\mu,\theta)$, we set
\begin{equation}
\label{2CauchySec}
\varphi(A)\,=\,\frac{1}{2\pi i}\,\int_{\partial\Sigma_\nu} \varphi(z) R(z,A)\, dz\,,
\end{equation}
where the boundary $\partial\Sigma_\nu$ is oriented counterclockwise. 
The sectoriality condition ensures that this integral is absolutely convergent 
and defines a bounded operator on $X$. Moreover by Cauchy's theorem, 
this definition does not depend on the choice of $\nu$. Further the resulting mapping 
\begin{equation*}
\begin{array}{cccc}
   &  H^\infty_0(\Sigma_\theta)   &  \longrightarrow   &  \mathcal{B}(X)  \\
   &   \varphi &  \longmapsto       &  \varphi(A)  \\
\end{array}
\end{equation*}
is an algebra homomorphism which is consistent with the usual functional calculus 
for rational functions.

We say that $A$ admits a bounded $H^\infty(\Sigma_\theta)$ functional 
calculus if the latter homomorphism is bounded, that is, there exists a constant $K \geq 0$ such that 
$$
\bnorm{\varphi(A)}_{X \to X} \leq K\norm{\varphi}_{H^\infty(\Sigma_\theta)},\qquad \text{for all } 
\varphi \in
H^\infty_0(\Sigma_\theta).
$$
If $A$ has dense range and admits a bounded $H^\infty(\Sigma_\theta)$ 
functional calculus, then the above homomorphism naturally extends to 
a bounded homomorphism $\varphi\mapsto \varphi(A)$ from the whole space
$H^\infty(\Sigma_\theta)$ into $\mathcal{B}(X)$.

For a sectorial operator $A$, the fractional powers  
$A^{\alpha}$ can be defined for any $\alpha>0$. We refer to \cite[Chapter 3]{Haa06}, 
\cite{KuW04} and \cite{CarSan01} for various definitions of 
these operators and their basic properties. 
The spectral mapping theorem for fractional powers states that  
\begin{equation} 
\label{efracsp} 
\sigma(A^{\alpha}) = \big\{ z^{\alpha} \co z \in \sigma(A) \big\},  
\;\;\; 
\alpha>0.   
\end{equation} 
If, in addition, the operator $A$ is bounded, then for any $\alpha>0$ the operator $A^{\alpha}$ is bounded
as well. 
We will frequently use the following effect of 
taking fractional powers on sectoriality and functional calculus.

\begin{lemma}
\label{fraction}
Let $A$ be sectorial of type $0< \mu <\pi$ and let $\alpha\in (0,\frac{\pi}{\mu})$. Then $A^\alpha$
is sectorial of type $\alpha\mu$. If further $A$ admits a bounded $H^\infty(\Sigma_\theta)$ 
functional calculus for some $\theta\in (\mu,\pi)$ and $\alpha<\frac{\pi}{\theta}$,
then $A^\alpha$
admits a bounded $H^\infty(\Sigma_{\alpha\theta})$ functional calculus.
\end{lemma}

We now turn to Ritt operators, the second key class of operators 
considered in this paper. We describe 
their holomorphic functional calculus and present some of their main features. There is now
a vast literature on this topic in which 
details and complements can be found, see in particular
\cite{ArhLeM11}, \cite{Arh11}, \cite{Blu01}, \cite{Blu01b}, 
\cite{Mer12},  
\cite{Lyu99}, \cite{NagZem99}, \cite{Nev93}, 
\cite{Vit05} and the references therein.

An operator $T \in \mathcal{B}(X)$ is called a 
Ritt operator if the two sets
\begin{equation}
\label{ensembles de Ritt} \big\{T^n\,:\, n\geq
0\big\}\qquad\hbox{and}\qquad \big\{n(T^n-T^{n-1}) \,:\, n\geq
1\big\}
\end{equation}
are bounded. One can show that this is equivalent to the spectral inclusion
\begin{equation}\label{spectral inclusion}
\sigma(T)\subset \ovl{\mathbb{D}}
\end{equation}
and the norm-boundedness of the set
\begin{equation}\label{ensemble de Ritt}
\big\{(z-1)R(z,T)\,:\,\vert z\vert>1\big\}.
\end{equation}

The boundedness of (\ref{ensemble de Ritt}) 
implies the existence of a constant $K\geq 0$ such that
$$
|z-1|\bnorm{R(z,T)}_{X \to X}\leq K
$$ 
whenever
$\Re(z)>1$. This implies that $\Id-T$ is a sectorial operator of type $<\frac{\pi}{2}$. 
In fact, it is known that a bounded operator 
$T \co X \to X$ is a Ritt operator if and only if 
\begin{equation}
\label{2Ritt}
\sigma(T) \subset \mathbb{D} \cup \{1\}
\qquad \text{and} \qquad
\Id-T \text{ is a sectorial operator of type }< \tfrac{\pi}{2},
\end{equation}
see e.g. \cite[Prop. 2.2]{Blu01b}.

We introduce the Stolz domains $B_\gamma$ as sketched in Figure 1. 
Namely, for any angle $0 < \gamma < \frac{\pi}{2}$, we let $B_\gamma$ be the 
interior of the convex hull of $1$ and the disc $D(0,\sin\gamma)$. 
We will use the fact that for any $0 < \gamma < \frac{\pi}{2}$ 
there exists a positive constant $C_{\gamma}$ such that
\begin{equation}
\label{Stolz}
\frac{|1-z|}{1-|z|} \leq C_{\gamma},\qquad \text{for all } z \in B_\gamma.
\end{equation}

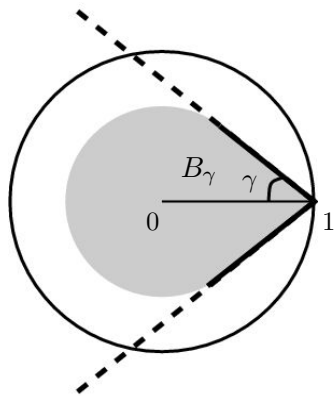
\begin{figure}[ht]
\vspace*{2ex}
\begin{center}
\includegraphics[scale=0.4]{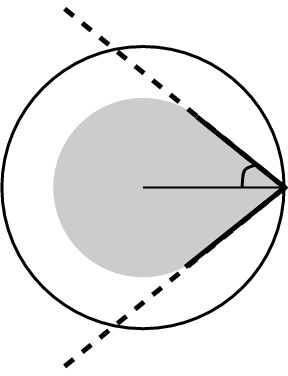}
\begin{picture}(0,0)
\put(-2,65){{\footnotesize $1$}}
\put(-68,65){{\footnotesize $0$}}
\put(-32,81){{\footnotesize $\gamma$}}
\put(-55,85){{\small $B_\gamma$}}
\end{picture}
\end{center}
\caption{\label{f1} Stolz domain}
\end{figure}

It is well-known that the spectrum of any Ritt operator $T$ is included in the 
closure of one of those Stolz domains. More precisely (see \cite[Lem. 2.1]{Mer12})
there exists $0<\beta<\frac{\pi}{2}$ such that 
\begin{equation}
\label{inclusion dans Stolz}
\sigma(T) \subset \ovl{B_\beta}
\end{equation}
and for any $\nu\in(\beta,\frac{\pi}{2})$, the set 
\begin{equation}
\label{2Ritt2}
\Big\{(z-1)R(z,T)\, :\, z\in\C \setminus \overline{B_\nu}\Big\}\qquad\hbox{is bounded}.
\end{equation}

The following is an analogue of the construction (\ref{2CauchySec}).
For any $\gamma\in (0,\frac{\pi}{2})$, let $H^\infty_0(B_\gamma)$ denote the space of all
bounded holomorphic functions $\phi\colon B_\gamma\to\C$ for which
there exist constants $s,K>0$ such that $\vert\phi(z)\vert\leq K\vert 1-z\vert^s$
for all $z\in B_\gamma$. Assume that $T$ satisfies (\ref{inclusion dans Stolz}) and
(\ref{2Ritt2}) for some $\beta<\gamma$. For $\phi\in H^\infty_0(B_\gamma)$ we set
\begin{equation}
\label{2CauchyRitt}
\phi(T)\,=\,\frac{1}{2\pi i}\,\int_{\partial B_\nu} \phi(z) R(z,T)\, dz\,,
\end{equation}
where $0<\beta<\nu<\gamma$ and the boundary $\partial B_\nu$ is oriented counterclockwise. 
This definition does not depend on $\nu$ and is consistent 
with the usual functional calculus
for polynomials. 

In accordance with ~\cite{Mer12}, we say that $T$ has a bounded $H^\infty(B_{\gamma})$ 
functional calculus if there exists a positive constant $K$ such that
$$
\bnorm{\phi(T)}_{X \to X} \leq
K\norm{\phi}_{H^\infty(B_{\gamma})}, \qquad\hbox{for all } 
\phi\in H^\infty_0(B_\gamma).
$$
Let $\mathcal{P}$ be the algebra of all 
complex polynomials.
We will use the following result (see \cite[Prop. 2.5]{Mer12}).

\begin{lemma}\label{pol}
A Ritt operator $T\co X\to X$ admits a bounded $H^\infty(B_{\gamma})$ 
functional calculus if (and only if) there exists a positive constant $K$ such that
$$
\bnorm{\phi(T)}_{X \to X} \leq
K\norm{\phi}_{H^\infty(B_{\gamma})}, \qquad\hbox{for all } \phi\in \mathcal{P}.
$$
\end{lemma}

The following result proved in ~\cite[Proposition 4.1]{Mer12} 
allows one to transfer known results from the theory of functional calculus for sectorial operators
to the context of Ritt operators. Recall property (\ref{2Ritt}).

\begin{prop}
\label{Prop LM4.1}
Let $T \co X \to X$ be a Ritt operator on a Banach space $X$. Then the following are equivalent.
\begin{itemize}
\item [(i)] $T$ admits a bounded $H^{\infty}(B_\gamma)$ functional calculus for some $\gamma \in (0,\frac{\pi}{2})$.
\item [(ii)] $\Id-T$ admits a bounded $H^{\infty}(\Sigma_\theta)$ functional calculus for some $\theta \in (0, \frac{\pi}{2})$.
\end{itemize}
\end{prop}

We now turn to background on compressions.
Let $X$ be a Banach space and let $T \co X \to X$ be a bounded operator. 
Let $F \subset E\subset X$ be closed subspaces. Assume that $E$ and $F$ are $T$-invariant 
i.e. we have $T(E) \subset E$ and $T(F) \subset F$. Then $T$ determines a bounded 
operator $\widetilde{T} \co E/F \to E/F$, called a compression of $T$. It is
characterized by the equality 
$$
\widetilde{T}q=qT|_{E},
$$ 
where $q \co E \to E/F$ denotes the canonical quotient map.
This clearly implies that for any complex polynomial $\phi$, we 
have $\phi(\widetilde{T})q=q\phi(T)|_{E}$. 
Note that a compression of a compression of $T$ is again 
a compression of $T$.

In the following, an algebraic semigroup 
$\mathscr{S}$ is a set supplied with an associative binary composition 
with identity $e$. A representation $\pi \co \mathscr{S} \to \mathcal{B}(X)$ 
is a map satisfying $\pi(st)=\pi(s)\pi(t)$ for any $s,t\in\mathscr{S}$ and $\pi(e) = 
\Id_X$. We say that a closed subspace $E\subset X$ is $\pi$-invariant
if it is $\pi(t)$-invariant for any $t\in\mathscr{S}$. If $F\subset E\subset X$ are
two $\pi$-invariant closed subspaces, we let $\widetilde{\pi}\colon
\mathscr{S}\to\mathcal{B}(E/F)$ be defined by $\widetilde{\pi}(t)=\widetilde{\pi(t)}$.
It is plain that $\widetilde{\pi}$ is a representation of $\mathscr{S}$. This
will be called a compressed representation in the sequel.

We will need the following proposition, which is a variant of \cite[Proposition 4.2]{Pis01}. 
See \cite[Proposition 5.5.6]{Fac15} for a proof.

\begin{prop}\label{compression}
Let $\mathscr{S}$ be an algebraic semigroup and let $\pi \co \mathscr{S} 
\to \mathcal{B}(X)$ and $\rho \co \mathscr{S} \to \mathcal{B}(Z)$ be 
representations of $\mathscr{S}$ on two Banach spaces 
$X$ and $Z$. Assume that $J \co X \to Z$ and $Q \co Z \to X$ are two bounded operators such that
$$
\pi(t)=Q\rho(t)J,\qquad \text{for all } t \in \mathscr{S}.
$$
Then $\pi$ is similar to a compression of $\rho$, that is, 
there exist $\rho$-invariant 
closed subspaces $F \subset E \subset Z$ and an isomorphism $S \co X \to E/F$ such that 
$\norm{S}_{}\norm{S^{-1}}_{} \leq \norm{Q}_{}\norm{J}_{}$ and  the compressed 
representation $\widetilde{\rho} \co  \mathscr{S}\to \mathcal{B}(E/F)$ satisfies
$$
\pi(t)=S^{-1}\widetilde{\rho}(t)S,  \qquad \text{for all } t \in \mathscr{S}.
$$
\end{prop}

We end this section with some background on complexifications and orders.
A complex Banach space $X$ is called the complexification
of a real Banach space $X_\R$ if $X_\R$ is a closed real subspace of
$X$, we have a real direct sum decomposition $X=X_\R\oplus iX_\R$
and for any $x_1,x_2$ in $X_\R$, 
\begin{equation}\label{2Complex}
\max\{\norm{x_1},\norm{x_2}\}\,\leq \norm{x_1 + ix_2}.
\end{equation}
In this situation, we say that an operator $T\in\mathcal{B}(X)$ is real 
if $T$ maps $X_\R$ into itself.

We say that $X$ is an ordered Banach space when it is 
the complexification
of a real Banach space $X_\R$ equipped with a partial
order $\geq$ compatible with the vector space structure, 
and such that the positive cone
$\Cc=\{x\in X_\R\, :\, x\geq 0\}$ is closed and proper, that is, 
$\Cc\cap(-\Cc)=\{0\}$.
An operator $T\in\mathcal{B}(X)$ on an ordered Banach space
is called positive if $T(x)\geq 0$ for any $x\geq 0$.

An ordered Banach space $X$ is called absolutely monotone if for any
$x,y\in X_\R$, we have
\begin{equation}\label{2Normal}
-y\leq x\leq y\,\Longrightarrow\, \norm{x}\leq\norm{y}.
\end{equation} 
Next $X$ is called a Riesz-normed space if moreover for any $x\in
X_\R$ and any $\epsilon>0$, there exists $y\in X_\R$ such that
$-y\leq x\leq y$ and $\norm{y}\leq (1+\epsilon)\norm{x}$. 
In this case any $x\in X_\R$ is the difference of two
positive elements. Indeed if $-y\leq x\leq y$, then $y\geq 0$,
$y-x\geq 0$ and $x = y-(y-x)$. Consequently, any 
positive $T\in\mathcal{B}(X)$ is real. We refer
e.g. to \cite{BR84} for more on these notions.

The class of Riesz-normed spaces includes Banach lattices, 
for which we refer to \cite{LinTza79}, 
and noncommutative $L^p$-spaces,
for which we refer to \cite{PX03}.


\section{Fractional Powers for Power Bounded Operators}\label{sec:fractional_powers}


Let $X$ be a Banach space. A bounded operator $T \co X \to X $ is 
power bounded if the set $\{T^n \ : \ n \geq 0\}$ is bounded in $\mathcal{B}(X)$. 
The spectrum of such an operator is contained in $\overline{\mathbb{D}}$ and
$\Id-T$ is sectorial of type $\frac{\pi}{2}$, see e.g. 
\cite[Lemma 3.1]{HT10}. Thus we may define
\begin{equation}\label{Talpha}
T_{\alpha} \coloneqq \Id - (\Id - T)^{\alpha}
\end{equation}
for any $\alpha>0$. By abuse of language, these operators $T_\alpha$ will be called
the {\it fractional powers of $T$}.

A power bounded operator is not necessarily a Ritt operator. 
However Dungey proved in \cite[Theorem 4.3]{Dun11} that $T_\alpha$ is 
a Ritt operator whenever $\alpha<1$. We give an elementary proof and complements 
below.

\begin{thm}
\label{thm:fractional_powers}
Let $X$ be a Banach space, let $T\co X \to X$ be a power bounded operator and let $\alpha \in (0,1)$. Then we have:
\begin{itemize}
\item[(a)] $T_\alpha$ is a Ritt operator.
\item[(b)] If $T$ is contractive, then $T_{\alpha}$ is contractive as well.
\item[(c)] If $X$ is an ordered Banach space and $T$ is positive, then $T_{\alpha}$ is positive as well.
\end{itemize}
\end{thm}

\begin{proof} We let $\alpha \in (0,1)$ and we set $g_\alpha(z) = 1-(1-z)^\alpha$ for any 
$z\in\overline{\D}$ (with the convention that $0^\alpha=0$). We have
\begin{equation*}
g_\alpha(z) = \sum_{k=1}^{\infty} a_{\alpha,k} z^k \qquad \text{for all } z \in \overline{\mathbb{D}},
\end{equation*}
where 
\begin{equation}\label{Dun1}
a_{\alpha,k} = (-1)^{k-1} \binom{\alpha}{k}>0
\qquad\hbox{and}\qquad
\sum_{k=1}^{\infty} a_{\alpha,k} = 1. 
\end{equation}

Using the compatibility of the holomorphic functional calculus with fractional powers
(see~\cite[Proposition 3.2]{HT10}), this implies that
\begin{equation}\label{Dun2}
T_{\alpha} = \Id - (\Id - T)^{\alpha} = \sum_{k=1}^{\infty} a_{\alpha,k} T^k.
\end{equation}
Next observe that for any $z \in \ovl{\mathbb{D}}$, we have
$$
\left|g_\alpha(z)\right|
=\left|\sum_{k=1}^{\infty} a_{\alpha,k} z^k\right|
\leq \sum_{k=1}^{\infty} a_{\alpha,k} |z|^k
\leq \sum_{k=1}^{\infty} a_{\alpha,k}= 1.
$$
This shows that $g_\alpha(\ovl{\mathbb{D}}) \subset \ovl{\mathbb{D}}$.
Moreover, it is easy to see that the above inequality is an equality (if and) 
only if $z=1$. Hence we deduce that
\begin{equation}
\label{inclusion galpha de D}
g_\alpha(\ovl{\mathbb{D}}) \subset \mathbb{D}\cup \{1\}.
\end{equation}
By (\ref{efracsp}), we have
$\sigma(T_{\alpha})=g_\alpha(\sigma(T))$. 
Moreover since $T$ is power bounded, we have $\sigma(T) \subset \ovl{\mathbb{D}}$. 
We deduce that 
$\sigma(T_{\alpha})\subset g_\alpha(\ovl{\mathbb{D}})$ and hence
$$
\sigma(T_\alpha) \subset \mathbb{D}\cup \{1\}.
$$
Moreover the fractional power $(\Id-T)^\alpha$ is sectorial of type $\frac{\alpha\pi}{2}<\frac{\pi}{2}$,
by Lemma \ref{fraction}. Applying (\ref{2Ritt}), we conclude that $T_\alpha$ is a Ritt operator.

This completes (a). The proofs of (b) and (c)  immediately follow from (\ref{Dun1}) and (\ref{Dun2}).
\end{proof}

We now consider the behaviour of $T_\alpha$ when $T$ is a Ritt operator and $\alpha>1$.

\begin{prop}
\label{prop:fractional_hinfty}
Let $T\co X \to X$ be a Ritt operator on a Banach space $X$.
\begin{itemize}
\item [(a)] For sufficiently small $\alpha >1$, $T_{\alpha}$ is a Ritt operator.
\item [(b)] Assume that $T$ admits a bounded $H^{\infty}(B_{\gamma})$ functional calculus 
for some $\gamma \in (0, \frac{\pi}{2})$. Then for sufficiently small $\alpha >1$, $T_{\alpha}$
admits a  bounded $H^{\infty}(B_{\gamma'})$ functional calculus 
for some $\gamma' \in (0, \frac{\pi}{2})$.
\end{itemize}
\end{prop}

\begin{proof}
Part (a) is already contained in~\cite[Theorem~1.3 (IV)]{Dun11}. 
For the sake of completeness, we recall the short argument. 
The operator $\Id-T$ is sectorial of type $<\frac{\pi}{2}$. Hence for $\alpha>1$ 
sufficiently close to 1, the operator $\Id-T_\alpha=(\Id-T)^\alpha$ 
is also sectorial of type $<\frac{\pi}{2}$, by Lemma  \ref{fraction}. Moreover, by (\ref{efracsp}) 
and (\ref{inclusion dans Stolz}), it is easy to see that
$$
\sigma(T_\alpha)=
\big\{1-(1-z)^\alpha \ :\ z\in \sigma(T) \big\} \subset \mathbb{D} \cup \{1\}
$$
for any $\alpha>1$ sufficiently close to 1. 
By (\ref{2Ritt}), we conclude that $T_{\alpha}$ is a Ritt operator.

To prove part (b), assume that
$T$ admits a bounded $H^{\infty}(B_{\gamma})$ functional calculus 
for some $\gamma \in (0, \frac{\pi}{2})$. It follows from Proposition~\ref{Prop LM4.1} 
that the operator $\Id-T$ admits a bounded $H^{\infty}(\Sigma_{\theta})$ 
functional calculus for some $\theta \in (0, \frac{\pi}{2})$. For sufficiently small $\alpha > 1$, it follows from Lemma \ref{fraction}
that $(\Id - T)^{\alpha}$ has a bounded $H^{\infty}(\Sigma_{\theta'})$ functional calculus for some 
$\theta'\in (0, \frac{\pi}{2})$. The converse implication of Proposition~\ref{Prop LM4.1} now implies the assertion.
\end{proof}

We now prove an $H^\infty$ functional calculus property on UMD Banach spaces. 
As is well-known, this class provides a natural setting for functional calculus and 
vector-valued harmonic analysis. We refer to \cite{Bur01} and the 
references therein for information. 
We simply note for further use that for any $p \in (1, \infty)$ and 
any measure space $\Omega$, the Bochner space $L^p(\Omega;X)$ is UMD 
if $X$ is UMD. Next, the UMD property is stable under passing to 
subspaces and quotients. In particular, $SQ_p$-spaces are UMD
for any $1<p<\infty$. Furthermore, UMD spaces are always reflexive
and the dual of any UMD space is UMD as well.

We say that an isomorphism $U \co X \to X$ on a Banach space $X$ is a power bounded isomorphism if
$$
\bigl\{U^n \ \co\ n \in \Z \bigr\}
$$ 
is bounded in $\mathcal{B}(X)$. In this case, $U_\alpha$ is a well-defined
Ritt operator for any $\alpha\in(0,1)$, by Theorem~\ref{thm:fractional_powers}.

\begin{thm}
\label{thm:fractional_power_hinfty}
Let $U$ be a power bounded isomorphism on a UMD Banach space $X$. Then for every $\alpha \in (0,1)$,
the operator $U_{\alpha}$ admits a bounded $H^{\infty}(B_{\gamma})$ functional calculus 
for some $\gamma \in (0,\frac{\pi}{2})$.
\end{thm}

\begin{proof} The proof heavily relies on some results and techniques from \cite{ArhLeM11} concerning the
shift operator. Following this paper, we consider 
$$
\mathbb{D}_\theta = \,D\bigl(-{\rm i}\cot(\theta),
\tfrac{1}{\sin(\theta)}\bigr) \cup D   
\bigl({\rm i}\cot(\theta),\tfrac{1}{\sin(\theta)}\bigr)
$$
for any $\frac{\pi}{2}<\theta<\pi$ (see Figure 2 below).

\begin{figure}[ht]
\begin{center}
\includegraphics[scale=0.5]{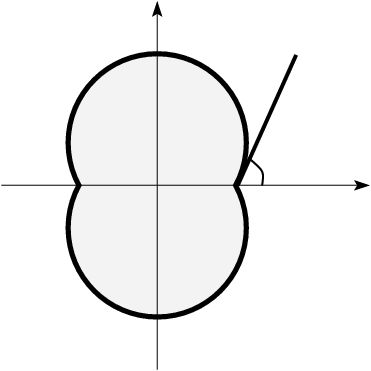}
\begin{small}
\begin{picture}(0,0) 
\put(-59,100){$\pi-\theta$} 
\put(-140,125){$\mathbb D_\theta$} 
\put(-155,83){$-1$} 
\put(-123,83){$0$} 
\put(-82,83){$1$} 
\end{picture}
\end{small}
\end{center}
\caption{\label{f2} Domain $\mathbb{D}_\theta$}
\end{figure}
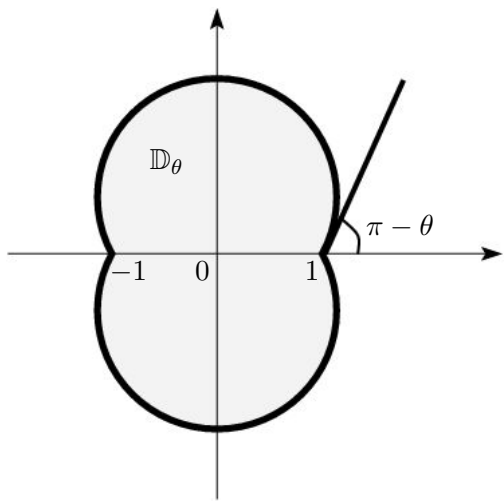

Let $\alpha\in(0,1)$ and let $\theta>\frac{\pi}{2}$ such that $\alpha\theta<\frac{\pi}{2}$.
Let $S_X \co \ell^p_{\mathbb{Z}}(X) \to \ell^p_{\mathbb{Z}}(X)$ 
denote the vector-valued shift operator given by 
$$
S_X\big((x_k)_{k \in \Z}\big) =(x_{k-1})_{k \in \Z}.
$$
Let $K=\sup_{n\in \mathbb{Z}}\norm{U^n}$. By the transference principle \cite[Theorem 2.8]{BGM89}, one has
$$
\bnorm{\varphi(U)}_{X \to X} \leq K^2 \bnorm{\varphi(S_X)}_{\ell^p_{\mathbb{Z}}(X) \to \ell^p_{\mathbb{Z}}(X)}
$$
for any rational function $\varphi$ with poles outside $\ovl{\mathbb{D}_\theta}$. 
Further, it follows from the proof of \cite[Proposition~6.2]{ArhLeM11} 
that there exists a constant $C>0$ such that
$$
\bnorm{\varphi(S_X)}_{\ell^p_{\mathbb{Z}}(X) \to \ell^p_{\mathbb{Z}}(X)} \leq C\,
\sup\bigl\{\vert \varphi(z)\vert\, :\, z\in \mathbb{D}_\theta\bigr\}
$$
for such functions. The existence of this constant $C$ depends on the UMD property of $X$.
Combining these two estimates yields
$$
\bnorm{\varphi(U)}_{X \to X} \leq C K^2\, 
\sup\bigl\{\vert \varphi(z)\vert\, :\, z\in \mathbb{D}_\theta\bigr\}
$$
for any rational function $\varphi$ with poles outside $\ovl{\mathbb{D}_\theta}$. 
Then the inclusion $1-\mathbb{D}_\theta \subset \Sigma_\theta$ and the method of the proof of \cite[Propositions~4.7 and 6.2]{ArhLeM11} imply that the sectorial operator $\Id - U$ admits a bounded $H^{\infty}(\Sigma_{\theta})$ functional calculus.
Applying Lemma \ref{fraction} we deduce that 
the  operator $(\Id - U)^{\alpha}$ has a bounded $H^{\infty}(\Sigma_{\alpha\theta})$ functional calculus.
Finally by Proposition~\ref{Prop LM4.1}, we obtain that
the Ritt operator $U_{\alpha} = \Id - (\Id - U)^{\alpha}$ has a bounded 
$H^{\infty}(B_{\gamma})$ functional calculus for some $\gamma \in (0, \frac{\pi}{2})$.
\end{proof}


\section{Positive isometric Dilations of Ritt Operators and analytic semigroups}\label{sec:dilations}


The aim of this section is to improve dilation results from \cite{FroWei06} and {\cite{ArhLeM11} concerning
bounded analytic semigroups or Ritt operators with a bounded $H^\infty$ functional calculus. 
The main point is that we are able to construct isometric dilations
whereas the above cited papers only established isomorphic dilations. 
Also the class of Banach spaces on which we consider semigroups or Ritt operators 
is larger than the ones in \cite{FroWei06,ArhLeM11}.

Our constructions will rely on abstract square functions. 
We start with some preliminaries on such objects in the discrete case. 
We let $\Omega_0=\{-1,1\}^\Z$ equipped with its normalized Haar measure. 
For any integer $k \in \Z$, we define $\epsi_k$ by $\epsi_k(\omega)=\omega_k$ 
if $\omega=(\omega_i)_{i \in \Z} \in \Omega_0$. The coordinate functionals $\epsi_k$
are independent Rademacher variables on the probability space $\Omega_0$.

Let $X$ be a Banach space and let $1<p<\infty$. 
We let $\Rad_p(X)\subset L^p(\Omega_0;X)$ be the closure of 
${\rm Span}\bigl\{\epsi_k \ot x\ |\ k \in \Z,\ x\in X\bigr\}$ in 
the Bochner space $L^p(\Omega_0;X)$. Thus, 
for any finite family  $(x_k)_{k \in \Z}$ of elements of $X$, we have
\begin{equation*}
\Bgnorm{\sum_{k \in \Z}^{} \epsi_k \ot x_k}_{\Rad_p(X)} 
\,=\,
\Bigg(\int_{\Omega_0}\bgnorm{\sum_{k \in \Z} \epsi_k(\omega)\,
x_k}_{X}^{p}\,d\omega\,\Bigg)^{\frac{1}{p}}.
\end{equation*}
We simply write $\Rad(X)=\Rad_2(X)$. 
By Kahane's inequalities (see e.g. \cite[Theorem 11.1]{DJT95}), 
the Banach spaces $\Rad(X)$ and $\Rad_p(X)$ are canonically isomorphic.

Let $T \co X \to X$ be a Ritt operator and let $\alpha>0$. For any $x \in X$ and
any $k \geq 0$, consider the element $x_k=(k+1)^{\alpha-\frac{1}{2}}T^{k}(\Id-T)^\alpha x$ of $X$. 
If the series $\sum_{k\geq 0}\varepsilon_k \ot x_k\,$ converges in $L^2(\Omega_0;X)$ then we set
$$
\norm{x}_{T,\alpha}\, =\,\biggnorm{\sum_{k=0}^{\infty}  (k+1)^{\alpha - 
\frac{1}{2}}\, \varepsilon_k\otimes T^{k}(\Id-T)^\alpha
x}_{\Rad(X)}.
$$
We set $\norm{x}_{T,\alpha}=\infty$ otherwise. These `square functions' $\norm{\ }_{T,\alpha}$
were introduced in \cite{ArhLeM11} to which we refer for more information (see also \cite{Mer12}).

In the sequel we consider Banach spaces with finite cotype. We refer
the reader e.g. to \cite{DJT95} for information on cotype. We note that 
if a Banach space $X$ is UMD, then $X$ and $X^*$ have finite cotype.

We also note that if $X=X_\R\oplus i X_\R$ is an ordered Banach space (see the last part
of Section 2), 
then $L^p(\Omega';X)$ is the complexification
of the real space $L^p(\Omega';X_\R)$. Moreover the latter 
has a natural order defined by writing for $f\in L^p(\Omega';X_\R)$
that $f\geq 0$ when $f(\omega)\geq 0$ for almost every $\omega\in\Omega'$.
This makes $L^p(\Omega';X)$ an ordered Banach space.

\begin{thm}
\label{thm:positive_dilation} Assume that $X$ is a reflexive Banach space and that
$X$ and $X^*$ have finite cotype.
Let $T \co  X \to X$ be a Ritt operator which admits a bounded $H^{\infty}(B_{\gamma})$ 
functional calculus for some $\gamma \in (0, \frac{\pi}{2})$. Let $1<p<\infty$. 
Then there exist a measure space $\Omega'$, an isometric isomorphism 
$U \co L^p(\Omega';X) \to L^p(\Omega';X)$ together with two 
bounded operators $J \co X \to L^p(\Omega';X)$ and 
$Q \co L^p(\Omega';X) \to X$ such that
$$
T^n = Q U^n J, \qquad \text{for all } n \ge 0.
$$
Moreover: 
\begin{itemize}
\item [(a)] If $X$ is an ordered Banach space, then the map $U$ can  be chosen to be a positive operator;
\item [(b)] If $X$ is a closed subspace of an $L^p$-space, then the Bochner space $L^p(\Omega';X)$
is also a closed subspace of an $L^p$-space and the map $U$ can  be chosen to be the restriction 
of a positive isometric isomorphism on an $L^p$-space;
\item [(c)] If $X$ is an $SQ_p$-space, then the Bochner space $L^p(\Omega';X)$ is also an $SQ_p$-space,
and the map $U$ can  be chosen to be a compression of a positive isometric isomorphism on an $L^p$-space.
\end{itemize}
\end{thm}

\begin{proof} 
We start with the construction of $U$, which is a universal operator
(it does not depend on $T$), and we check its properties listed in (a), (b) and (c).
First we define $\mathfrak{u} \co L^p(\Omega_0) \to L^p(\Omega_0)$ as the 
pullback of the coordinate right shift. That is, for any $f\in L^p(\Omega_0)$ and any
$(\omega_k)_{k \in \mathbb{Z}} \in \Omega_0$, 
$$
\mathfrak{u}(f)\big((\omega_k)_k\big)=f\big((\omega_{k-1})_k\big). 
$$
Then $\mathfrak{u}$ is a positive isometric isomorphism. Hence
one can extend $\mathfrak{u} \ot \Id_X$ to an isometric 
isomorphism $\mathfrak{U}$ on the Bochner space $L^p(\Omega_0;X)$.

Note that for any $k \in \Z$, one  has $\mathfrak{u}(\epsi_k) = \epsi_{k-1}$. 
Consequently, for any element $\sum_{k \in \Z} \epsi_k \ot x_k $ in $\Rad_p(X)$, 
we have
\begin{equation}\label{Action}
\mathfrak{U}\bigg(\sum_{k \in \Z} \epsi_k \ot x_k\bigg)=\sum_{k \in \Z} 
\epsi_{k-1} \ot x_k=\sum_{k \in \Z} \epsi_{k} \ot x_{k+1}.
\end{equation}
The $\ell^p$-direct sum $X \oplus_{p} L^p(\Omega_0;X)$ is 
canonically isometrically isomorphic
to a Bochner space $L^p(\Omega';X)$. 
We let $U$ be the isometric isomorphism $\Id_X \oplus \mathfrak{U}$ on 
the space $X \oplus_{p} L^p(\Omega_0;X) =  L^p(\Omega';X)$. 

If $X$ is an ordered Banach space, then 
$U$ is clearly a positive operator. 
If $X$ is a closed subspace of an $L^p$-space $L^p(\Omega)$, 
then $U$ is the restriction of the positive isometry 
$\Id_{L^p(\Omega)} \oplus (\mathfrak{u} \overline{\ot} \Id_{L^p(\Omega)})$ 
on the $L^p$-space $L^p(\Omega) \oplus_{p} 
L^p(\Omega_0;L^p(\Omega))=L^{p}(\Omega' \times \Omega)$. 
If $X=E/F$ for some closed subspaces $F\subset E\subset L^p(\Omega)$, then
$L^p(\Omega';F)\subset L^p(\Omega';E)$ are closed subspaces of $L^p(\Omega';L^p(\Omega))=L^p(\Omega' 
\times \Omega)$ and, by \cite[Proposition 7.4]{Defant-Book}, we have an isometric isomorphism
$$
L^p(\Omega';X)=L^p(\Omega';E/F)=L^p(\Omega';E)/L^p(\Omega';F).
$$
Hence the Bochner space $L^p(\Omega',X)$ is a quotient of a closed subspace 
of $L^p(\Omega' \times \Omega)$. 
Moreover $U$ is the compression of the 
positive isometric isomorphism $ \Id_{L^p(\Omega)} \oplus (\mathfrak{u} 
\ot \Id_{L^p(\Omega)}) \co L^p(\Omega' \times \Omega) \to L^p(\Omega' 
\times \Omega)$ with respect to the subspaces $L^p(\Omega',F) \subset 
L^p(\Omega',E) \subset L^p(\Omega' \times \Omega)$.

We now show that $U$ is a dilation of $T$. 
Since the Banach space $X$ has finite cotype, we obtain from the assumption and 
\cite[Theorem 6.4]{Mer12} that the operator $T$ admits a quadratic $H^\infty(B_{\gamma'})$ 
functional calculus for some $\gamma' \in (\gamma,\frac{\pi}{2})$. That is, 
there exists a positive constant $C$ such that for any integer $n \geq 0$, 
for any $\phi_0,\ldots,\phi_n$ in $H_{0}^{\infty}(B_{\gamma'})$ and for any $x \in X$,
$$
\bgnorm{\sum_{k=0}^{n} \epsi_k \ot \phi_k(T)x}_{{\rm Rad}(X)}\leq C \norm{x}_X\,
\sup\biggl\{\bigg(\sum_{k=0}^{n} |\phi_k(z)|^2\bigg)^{\frac{1}{2}}\, :\, z\in B_{\gamma'}\biggr\}.
$$
Let us apply this estimate with 
$\phi_k(z)=z^{k}(1-z)^{\frac{1}{2}}$. For any $z \in B_{\gamma'}$, we have
$$
\sum_{k=0}^{\infty} |\phi_k(z)|^2 
=\sum_{k=0}^{\infty} |z|^{2k} |1-z| 
= \frac{|1-z|}{1-|z|^2}\leq \frac{|1-z|}{1-|z|}\leq C_{\gamma'},
$$
by (\ref{Stolz}). We deduce that for any $n\geq 1$,
$$
\biggnorm{\sum_{k=0}^{n}  \varepsilon_k\otimes T^{k}(\Id-T)^{\frac12}
x}_{\Rad(X)}\,\leq C C_\gamma\norm{x}_X.
$$
Since $X$ has finite cotype, this uniform estimate implies the convergence 
of the series $\sum_{k}  \varepsilon_k\otimes T^{k}(\Id-T)^{\frac12}
x$ (see \cite{Kwa}). Then we deduce (with $C'=CC_\gamma$) 
the following square function estimate,
$$
\norm{x}_{T,\frac{1}{2}} \leq C'\norm{x}_X, \qquad \text{for all }  x \in X.
$$
Similarly, there is a positive constant $C''$ such that
$$
\norm{y}_{T^*,\frac{1}{2}}\leq C'' \norm{y}_{X^*}, \qquad \text{for all } y \in X^*.
$$

Since $T$ is power bounded and $X$ is reflexive, the Mean Ergodic Theorem 
(see e.g. \cite[Subsection 2.1.1]{Kr85}) ensures that we have  direct sum decompositions 
$$ X = \Ker (\Id-T) \oplus \overline{\Ran (\Id - T)}
\qquad\hbox{and}\qquad
X^* = \Ker (\Id-T^*) \oplus \overline{\Ran (\Id-T^*)}. 
$$
Furthermore for any $x_0\in \Ker (\Id-T), x_1\in \overline{\Ran (\Id - T)},
y_0\in \Ker (\Id-T^*)$ and $y_1\in \overline{\Ran (\Id - T^*)}$, 
we have $\langle x_0,y_1\rangle = \langle x_1,y_0\rangle =0$, hence
\begin{equation}\label{decomp}
\langle x_0 +x_1 , y_0 + y_1 \rangle
\,=\, \langle x_0,y_0\rangle + \langle x_1,y_1\rangle.
\end{equation}

Let $p'$ be the conjugate of $p$. Note that since $X$ is reflexive we have an isometric isomorphism
\begin{equation}\label{RNP}
L^p(\Omega;X)^* \,=\, L^{p'}(\Omega;X^*).
\end{equation}

From the above square function estimates, we  may define a bounded linear map
\begin{align*}
J_1 \co X =  \Ker (\Id-T) \oplus \ovl{\Ran(\Id-T)}& \to X \oplus_p L^p(\Omega_0;X) \\
x_0+x_1 & \mapsto \biggl(x_0, \sum_{k=0}^{\infty} \epsi_k \ot T^{k}(\Id-T)^{\frac{1}{2}}x_1\biggr)
\end{align*}
and a similar $J_2 \co X^{*} \to   X^* \oplus_{p'} L^{p'}(\Omega_0;X)$. 
Consider $x_0 \in \Ker (\Id-T) $, $x_1 \in \ovl{\Ran(\Id-T)}$, $y_0\in \Ker (\Id-T^*)$ and 
$y_1\in \ovl{\Ran(\Id-T^*)}$. 
For any $n\geq 0$, we have
$$
U^n J_1(x_0+ x_1) = 
\biggl(x_0, \sum_{k=-n}^{\infty} \epsi_k \ot T^{k+n}(\Id-T)^{\frac{1}{2}}x_1\biggr),
$$
by (\ref{Action}). Hence
\begin{align*}
\bigl\langle U^n & J_1(x_0+ x_1) , J_2(y_0+y_1)\bigr\rangle
\\ & = \biggl\langle
\biggl(x_0, \sum_{k=-n}^{\infty} \epsi_k \ot T^{k+n}(\Id-T)^{\frac{1}{2}}x_1\biggr),
\biggl(y_0, \sum_{k=0}^{\infty} \epsi_k \ot T^{*n}(\Id-T^*)^{\frac{1}{2}}y_1\biggr)\biggr\rangle \\
& =
\langle x_0, y_0\rangle\, +\,
\sum_{k=0}^\infty\bigl\langle T^{k+n}(\Id-T)^{\frac{1}{2}}x_1, T^{*n}(\Id-T^*)^{\frac{1}{2}}y_1\bigr\rangle.
\end{align*}
As in the proof of~\cite[Theorem~4.8]{ArhLeM11} one shows, using the fact that
$x_1 \in \overline{\Ran(\Id-T)}$ and $y_1 \in \overline{\Ran(\Id-T^*)}$, that we have
$$
\sum_{k=0}^\infty\bigl\langle T^{k+n}(\Id-T)^{\frac{1}{2}}x_1, T^{*n}(\Id-T^*)^{\frac{1}{2}}y_1\bigr\rangle
= \bigl\langle 
(\Id + T)^{-1}T^n x_1,y_1\bigr\rangle.
$$
Now we introduce the bounded operator 
\begin{align*}
\Theta \co X =  \Ker (\Id-T) \oplus \ovl{\Ran(\Id-T)}& \to \Ker (\Id-T) \oplus \ovl{\Ran(\Id-T)}=X \\
x_0 + x_1  & \mapsto  x_0+ (\Id+T)x_1.
\end{align*}
Then it follows from results above that
$$
\bigl\langle U^n J_1\Theta(x_0+ x_1) , J_2(y_0+y_1)\bigr\rangle = 
\langle x_0, y_0\rangle\, +\,\langle T^n x_1, y_1\rangle.
$$
Using (\ref{RNP}), we define $Q=J_2^*\colon X \oplus_p L^p(\Omega_0;X)\to X$. Then letting 
$J=J_1\Theta$, the above identity and (\ref{decomp}) yield $T^n=QU^nJ$ for any $n\geq 0$. 
\end{proof}

We now pass to semigroups. Roughly speaking, Rademacher averages used in the proof of 
Theorem \ref{thm:positive_dilation} will be replaced by deterministic stochastic 
integration with respect to some Brownian motion. We need some preliminaries on 
second quantization and on the vector valued Gaussian spaces $\gamma(\cdotp,X)$ 
introduced by Kalton-Weis in \cite{KW04} (a first version of this paper was circulated 
in 2001). These $\gamma$-spaces were introduced in order to define abstract square functions in 
the context of $H^\infty$ functional calculus. Since then, they were used in 
various other directions, see in particular \cite{Lem10} and 
\cite{vNeWei05,NeeVerWei13}. 
Note that in the present paper we need to mix real 
valued Gaussian variables and complex Banach spaces.

We start with a little background on Gaussian Hilbert spaces and second quantization.  
For more systematic discussions on this topic we refer
to the books \cite{Jan97} and \cite{Sim74}. Let $H_\R$ be a real Hilbert space.
A Gaussian random process indexed by $H_\R$ is a probability space 
$\widehat{\Omega}$ together with a linear isometry 
\begin{equation}\label{W}
W \co H_\R \longrightarrow L^2_\R(\widehat{\Omega})
\end{equation}
satisfying the following two properties. 
\begin{itemize}
\item [(i)] Each $W(h)$ is a Gaussian random variable;
\item [(ii)] The linear span of the products $W(h_1)W(h_2)\cdots W(h_m)$, 
with $m\geq 0$ and $h_1,\ldots, h_m$ in $H_\R$,
is dense in the real Hilbert space $L^2_\R(\widehat{\Omega})$.
\end{itemize}
Here we make the convention that the empty product, correspondong to $m=0$ in (ii), 
is the constant function $1$.
Each product $W(h_1)W(h_2)\cdots W(h_m)$ belongs to $L^p_\R(\widehat{\Omega})$ 
for any $1\leq p<\infty$
and their linear span is also dense in $L^p_\R(\widehat{\Omega})$.

Let $T\colon H_\R\to H_\R$ be a contraction. The `second quantization
of $T$' is a positive operator $\Gamma(T)\colon L^1(\widehat{\Omega})\to 
L^1(\widehat{\Omega})$ such that $\Gamma(T)(1) = 1$, 
\begin{equation}\label{SQ1}
\Gamma(T)W(h) = W\bigl(T(h)\bigr), \qquad\hbox{for all } h\in H_\R,
\end{equation}
and for any $1\leq p<\infty$, 
$\Gamma(T)$ restricts to a contraction 
$$
\Gamma_p(T)\colon L^p(\widehat{\Omega})\longrightarrow L^p(\widehat{\Omega}).
$$
Furthermore, the second quantization functor $\Gamma$ satisfies the following.

\begin{lemma}\label{SQ2} Let $1\leq p<\infty$.
\begin{itemize}
\item [(i)] For any two contractions $T_1,T_2\colon H_\R\to H_\R$, we
have $\Gamma_p(T_1 T_2) =  \Gamma_p(T_1)\Gamma_p(T_2)$.
\item [(ii)] If $(T_t)_{t \geq 0}$ is a $C_0$-semigroup of contractions 
on $H_\R$, then $\big(\Gamma_p(T_t)\big)_{t \geq 0}$ is a 
$C_0$-semigroup of contractions on $L^p(\widehat{\Omega})$.
\end{itemize}
\end{lemma}
These properties can be found in \cite[Chap. 4]{Jan97}, except the assertion (ii)
from Lemma \ref{SQ2}. In the latter statement, the semigroup property
follows from (i) and strong continuity of
$t\mapsto \Gamma_p(T_t)$ follows from the argument in \cite[Thm. 4. 20]{Jan97}.
Note that in this construction, $T$ acts on the {\it real} Banach space $H_\R$ 
whereas $\Gamma_p(T)$ acts on
the {\it complex} Banach space $L^p(\widehat{\Omega})$.

In the sequel we let $H$ denote the standard complexification of $H_\R$. 
For convenience we keep the notation $W$ to denote the 
complexification $W\colon H\to L^2(\widehat{\Omega})$ of
(\ref{W}). This is an isometry. 
Let $(g_n)_{n\geq 1}$ be an independent sequence of 
real valued standard Gaussian variables on some probability space.

Let $X$ be a complex Banach space and
let $H^*$ denote the dual space of $H$. We identify the
algebraic tensor product $H \ot X$ with the subspace of
$\mathcal{B}(H^*,X)$ of all bounded finite rank operators in the usual way.
Namely for any  $h_1,\ldots,h_n$ in $H$ and any $x_1,\ldots,x_n$ in $X$, 
we identify the element $\sum_{k=1}^n h_k\ot x_k$ with the operator
$u\colon H^*\to X$ defined by $u(\xi) = \sum_{k=1}^n \xi(h_k)\, x_k\,$
for any $\xi\in H^*$. For any $u \in H \ot X$, there exists a finite orthonormal
family $(e_1,\ldots,e_n)$ in the real space 
$H_\R$ and a finite family $(x_1,\ldots,x_n)$ of $X$ such that 
$u=\sum_{k=1}^n  e_k\ot x_k$. Then for any $1\leq p<\infty$, we set
$$
\norm{u}_{\gamma^p(H^*,X)}=\Bigg(\mathbb{E}
\bgnorm{ \sum_{k=1}^n g_k \ot x_k}_{X}^p\Bigg)^{\frac{1}{p}}.
$$
By \cite[Corollary 12.17]{DJT95} and its proof, this definition does not depend on the $e_k$'s and $x_k$'s 
representing $u$. We let $\gamma^p(H^*,X)$ be the completion of $H \ot X$ with 
respect to this norm. The identity mapping on $H \ot X$ extends to an injective and 
contractive embedding of $\gamma^p(H^*,X)$ into $\mathcal{B}(H^*,X)$. 
We may thus identify $\gamma^p(H^*,X)$ with a linear subspace in $\mathcal{B}(H^*,X)$. 
Note that by the Khintchine-Kahane inequalities, $\gamma^p(H^*,X)$ does not depend on 
$p$ as a linear space. 

For any orthonormal family $(e_1,\ldots,e_n)$ of $H_\R$, the $n$-tuple
$(W(e_1),\ldots, W(e_n))$ is an orthonormal family as well and all linear combinations
of the $W(e_i)$ are Gaussian. Hence $(W(e_1),\ldots, W(e_n))$ is an independent family of 
real valued standard Gaussian variables. Therefore the operator
$W\otimes \Id_X\colon H\otimes X\to L^p(\widehat{\Omega})\otimes X$
extends to an isometry
$$
W_{p,X}\colon \gamma^p(H^*,X)\longrightarrow L^p(\widehat{\Omega};X).
$$
We record for further use 
the following tensor extension property (see \cite[Prop. 4.3]{KW04}
and its proof).

\begin{lemma}\label{ideal} 
Let $T\colon H\to H$ be a real operator. Then for any $1\leq p<\infty$,
$T\otimes \Id_X\colon H\otimes X\to H\otimes X$
extends to a bounded operator
$$
M_{p,T}\colon \gamma^p(H^*,X)\longrightarrow \gamma^p(H^*,X),
$$
with $\norm{M_{p,T}}=\norm{T}$. Furthermore, if we regard 
$u\in\gamma^p(H^*,X)$ as an operator from $H^*$ into $X$, then 
$M_{p,T}(u)= u\circ T^*$.
\end{lemma}

We now consider the special case when $H=L^2(\Omega)$ for some measure
space $(\Omega,\mu)$ and $H_\R=L^2_\R(\Omega)$. We identity $L^2(\Omega)^*$
with $L^2(\Omega)$ through the standard duality pairing
$$
\langle h',h\rangle =\int_\Omega h'(s)h(s)\, d\mu(s),
$$
so that the above construction leads to $\gamma^p(L^2(\Omega),X)$
and yields an isometric embedding of that space into 
$L^p(\widehat{\Omega};X)$.

Let $f \co \Omega \to X$  be a measurable function which 
is weakly $L^2$, that is, $x^* \circ f \in L^2(\Omega)$ for any $x^* \in X^*$. 
Following \cite[Section 4]{KW04} we can define 
a bounded operator $u^f \co L^2(\Omega) \to X$ by
$$
\big\langle x^*, u^f(h)\big\rangle_{X^*,X}=
\int_{\Omega} \big\langle x^*, f(s)\big\rangle_{X^*,X}h(s) \,d\mu(s),
\qquad x^*\in X^*.
$$
Then we let $\gamma^p(\Omega,X)$ denote the space of all functions $f$ such that
$u^f\in \gamma^p(L^2(\Omega),X)$, equipped with the induced norm (that is, 
$\norm{f}_{\gamma^p(\Omega,X)} : =\norm{u_f}_{\gamma^p(L^2(\Omega),X)})$. 
Note that $\gamma^p(\Omega,X)$ contains all finite rank operators.
More precisely for any $h\in L^2(\Omega)$ and any $x\in X$, $h\otimes x = u^f$, where 
$f\colon \Omega\to X$ is defined by $f(s) = h(s)x$.

We will need the following duality result. 
Note that if $1<p<\infty$ and $p'$ denotes its conjugate, then
$L^{p'}(\widehat{\Omega};X^*)$ is a subspace of the dual space of 
$L^p(\widehat{\Omega};X)$, so that it
makes sense to consider the duality pairing 
$\langle\,\cdotp\,,\,\cdotp\,\rangle_{L^{p'}(\widehat{\Omega};X^*),
L^p(\widehat{\Omega};X)}$.

\begin{lemma}\label{duality}
Let $1<p,p'<\infty$ be conjugate numbers. Let
$f\in\gamma^p(\Omega,X)$ and $g\in\gamma^{p'}(\Omega,X^*)$.
Then $s\mapsto \langle g(s),f(s)\rangle_{X^*,X}$ 
belongs to $L^1(\Omega)$ and we have
$$
\int_\Omega \langle g(s),f(s)\rangle_{X^*,X}\,d\mu(s)\, =\, \bigl\langle W_{p',X^*}(u^g), W_{p,X}(u^f)
\bigr\rangle_{L^{p'}(\widehat{\Omega};X^*),
L^p(\widehat{\Omega};X)}.
$$
\end{lemma}

\begin{proof} The integrability of $\langle g(\cdotp),f(\cdotp)\rangle$ is established in 
\cite[Cor. 5.5]{KW04}. Moreover it is shown in that paper that 
for any $f\in\gamma^p(\Omega,X)$ and $g\in\gamma^{p'}(\Omega,X^*)$, the composition
operator ${u^g}^* u^f\colon L^2(\Omega)\to L^2(\Omega)$ is trace class and
$$
{\rm tr}\bigl({u^g}^* u^f\bigr) = \int_\Omega \langle g(s),f(s)\rangle_{X^*,X}\,d\mu(s)\,.
$$
Moreover 
\begin{equation}\label{continuity}
\bigl\vert {\rm tr}\bigl({u^g}^* u^f\bigr)\bigr\vert\leq
\norm{f}_{\gamma^p(\Omega,X)}\norm{g}_{\gamma^{p'}(\Omega,X^*)}.
\end{equation}
It therefore suffices to show that 
\begin{equation}\label{trace}
{\rm tr}\bigl({u^g}^* u^f\bigr) \,=\, 
\bigl\langle W_{p',X^*}(u^g), W_{p,X}(u^f)
\bigr\rangle_{L^{p'}(\widehat{\Omega};X^*),
L^p(\widehat{\Omega};X)}.
\end{equation}
Let $(u_k)_k$ and $(u'_k)_k$ be sequences in $L^2(\Omega)\otimes X$ and 
$L^2(\Omega)\otimes X^*$ respectively, such that 
$u_k\to u^f$ in $\gamma^p(L^2(\Omega),X)$ and 
$u'_k\to u^g$ in $\gamma^{p'}(L^2(\Omega),X^*)$, when $k\to\infty$.
Then by (\ref{continuity}), ${\rm tr}\bigl({u'_k}^* u_k\bigr)\to 
{\rm tr}\bigl({u^g}^* u_f\bigr)$. By the continuity of $W_{p',X^*}$ and 
$W_{p,X}$ we also have that $\bigl\langle W_{p',X^*}(u'_k), W_{p,X}(u_k)
\bigr\rangle\to \bigl\langle W_{p',X^*}(u^g), W_{p,X}(u^f)
\bigr\rangle$. Hence it suffices to show (\ref{trace}) in the finite rank case.
By linearity, we are reduced to check this identity when $u=u^f$ and
$u'=u^g$ are rank one.  

To proceed we let $h,h'\in L^2(\Omega)$, $x\in X, x^*\in X^*$ and consider
$u=h\ot x$ and $u'=h'\ot x^*$. On the one hand, $u'^*u = \langle x^*,x\rangle h\ot h'$
hence 
$$
{\rm tr}(u'^*u) = \langle x^*,x \rangle \int_\Omega h(s)h'(s)\, d\mu(s)\,.
$$
On the other hand, by the definition of $W_{p,X}$ and $W_{p',X^*}$, we have
$$
\langle W_{p',X^*}(u'), W_{p,X}(u)\rangle = \langle x^*,x \rangle
\int_{\widehat{\Omega}} W(h')W(h)\,.
$$
Since $W\colon L^2(\Omega)\to L^2(\widehat{\Omega})$ is an isometry, 
the right-hand side is equal
to 
$$\langle x^*,x \rangle \int_{\Omega} h'h\,.
$$ 
This shows 
(\ref{trace}) in that special case, and hence for any 
$f\in\gamma^p(\Omega,X)$ and $g\in\gamma^{p'}(\Omega,X^*)$
by the preceding reasoning.
\end{proof}

\begin{thm}
\label{Th dilation Delta semigroup}
Assume that $X$ is a reflexive Banach space and that
$X$ and $X^*$  have finite cotype. 
Let $A$ be a sectorial operator which admits a bounded $H^{\infty}(\Sigma_{\theta})$ 
functional calculus for some $\theta \in (0, \frac{\pi}{2})$. Let $1<p<\infty$. 
Then there exists a measure space $\Omega'$, a $C_0$-group $(U_t)_{t \in \R}$ of 
isometries on the Banach space $L^p(\Omega';X)$ together with two 
bounded operators $J \co X \to L^p(\Omega';X)$ and 
$Q \co L^p(\Omega';X) \to X$ such that
$$
e^{-tA} = Q U_t J, \qquad \text{for all } t \ge 0.
$$
Moreover: 
\begin{itemize}
\item [(a)] If $X$ is an ordered Banach space, then the maps $U_t$ can  be chosen to be positive operators;
\item [(b)] If $X\subset L^p(\Omega)$ is a closed subspace of an $L^p$-space, then 
$(U_t)_{t\in \R}$ is the restriction of  
a $C_0$-group $(V_t)_{t\in\R}$ of positive isometries on $L^p(\Omega'\times\Omega)$.
\item [(c)] If $X$ is an $SQ_p$-space, then there exists a measure space $\Omega$, a $C_0$-group 
$(V_t)_{t\in\R}$ of
positive isometries on $L^p(\Omega'\times\Omega)$ and
two closed subspaces $F \subset E \subset L^p(\Omega' \times \Omega)$ 
which are invariant under each $V_t$, such that $L^p(\Omega';X) = E/F$ and for any $t\in \R$,
$U_t$ is the compression of $V_t$ to $E/F$.
\end{itemize}
\end{thm}

\begin{proof}
The scheme of proof is similar to the one of Theorem \ref{thm:positive_dilation}. 
We start with the definition 
of $(U_t)_{t\in\R}$, which is a universal $C_0$-group. 
We apply the preceding construction to the case when $\Omega=\R$, equipped with 
the Lebesgue measure. Note that $W(\textbf{1}_{[0,t]})_{t\geq 0}$ is a Brownian motion.
For any $t \in\R$, we let $\tau_t\colon L^2(\R)\to L^2(\R)$ denote the shift operator defined by
$$
\tau_t(h) = h(\cdotp +t), \qquad h \in L^2(\R).
$$
This is a real operator and $(\tau_t)_{t\in \R}$  is a $C_0$-group of isometries 
on the real Hilbert space $L^2_\R(\R)$. Hence by Lemma \ref{SQ2}, 
$\bigl(\Gamma_p(\tau_t)\bigr)_{t\in \R}$ is a $C_0$-group of positive isometries
on $L^p(\widehat{\Omega})$. By positivity, $\Gamma_p(\tau_t)\ot \Id_X$ extends to an
isometry $V_t\colon L^p(\widehat{\Omega};X)\to L^p(\widehat{\Omega};X)$ for any 
$t\in\R$. Furthermore, $(V_t)_{t\in \R}$ is a $C_0$-group on that space.

From (\ref{SQ1}) we see that $\Gamma_p(\tau_t)W = W\tau_t$ for any $t$. Taking tensor extensions
and applying Lemma \ref{ideal}, we obtain that 
$$
\forall\, t\in\R,\qquad V_t W_{p,X} = W_{p,X}M_{p,\tau_t}.
$$
Suppose that $f\in\gamma^p(\R,X)$. Using the property $\tau_t^* = \tau_{-t}$ it is 
easy to see that $f(\cdotp +t)$ also belongs to $\gamma^p(\R,X)$ and that 
$M_{p,\tau_t}(u^f) = u^{f(\cdotp+t)}$. Combining with the above identity, we 
deduce that
\begin{equation}\label{shift}
\forall\, t\in\R,\ \forall\, f\in \gamma^p(\R,X),
\qquad 
V_t W_{p,X}(u^f) = W_{p,X}(u^{f(\cdotp+t)}).
\end{equation}
We set $U_t=\Id_X\oplus V_t\colon X\oplus_p L^p(\widehat{\Omega};X)
\to X\oplus_p L^p(\widehat{\Omega};X)$ for any $t\in\R$. Then 
$(U_t)_{t \in \R}$ is a $C_0$-group of 
isometries, and $X\oplus_p L^p(\widehat{\Omega};X)$ can be canonically
identified with a space $L^p(\Omega';X)$. The arguments to show that 
$(U_t)_{t \in \R}$ satisfies the `moreover part' of the statement are
similar to the ones in the proof of Theorem \ref{thm:positive_dilation}.

In the sequel we write $T_s=e^{-sA}$ for any $s\geq 0$.
Let $\theta'\in(\theta,\frac{\pi}{2})$. 
Since $X$ has finite cotype and $A$ admits a bounded $H^\infty(\Sigma_\theta)$
functional calculus, there exists, for any $\varphi\in H^\infty_0(\Sigma_{\theta'})$,
a positive constant $C$ such that for any $x\in X$, the function
$s\mapsto \varphi(sA)x$ belongs to $\gamma^p\bigl(L^2(\R_+^*,dt/t),X\bigr)$ and 
$$
\bnorm{s \mapsto \varphi(sA)x}_{\gamma^p(\R_+^*,X)}\leq C \norm{x}_{X}.
$$
This fundamental result is due to Kalton-Weis (see \cite[Section 7]{KW04}). 
In the latter paper only the case $p=2$ is treated. However since 
$\gamma^p(\cdotp,X)$ and $\gamma^2(\cdotp,X)$ are the same space with equivalent norms,
this special case implies the general case.
A proof can also be obtained by mimicking 
the one of \cite[Theorem 7.6]{JMX}, which corresponds 
to the case when $X$ is a 
noncommutative $L^p$-space. 
In the case 
when $X$ is an $L^p$-space, this estimate 
reduces to \cite[Section 6]{CDMY96}.

We apply this result with the function $\varphi$
defined by $\varphi(z) = z^{\frac12}e^{-z}$, which belongs to 
$H^\infty_0(\Sigma_{\theta'})$ for any $\theta'<\frac{\pi}{2}$. In this
case, $\varphi(sA)x = s^{\frac12}A^{\frac12} T_s x$. Let
$\chi = \chi_{(0,\infty)}$ on $\R$. It is easy to check that
$\norm{s \mapsto \varphi(sA)x}_{\gamma^p(\R_+^*,X)} = 
\norm{s \mapsto \chi(s)A^{\frac12} T_s x}_{\gamma^p(\R,X)}$. 
It therefore follows from above that we have an estimate
\begin{equation}\label{SFE}
\bnorm{s \mapsto \chi(s)A^{\frac12} T_s x}_{\gamma^p(\R,X)}\leq C \norm{x}_{X},
\qquad x\in X.
\end{equation}
We have a similar estimate on the dual space $X^*$, 
$$
\bnorm{s \mapsto \chi(s)A^{*\frac12} T_s^* y}_{\gamma^{p'}(\R,X^*)}
\leq C \norm{y}_{X^*}, \qquad y\in X^*.
$$

Since $X$ is reflexive, we have  direct sum decompositions 
$$
X= \Ker(A) \oplus \ovl{\Ran(A)}
\qquad\hbox{and}\qquad
X^*= \Ker(A^*) \oplus \ovl{\Ran(A^*)},
$$
see e.g. \cite[Proposition 2.1.1]{Haa06}.

Using (\ref{SFE}) one can define a bounded linear map
\begin{align*}
J_1 \co X =  \Ker (A) \oplus \ovl{\Ran(A)} & \longrightarrow X\oplus_p L^p(\widehat{\Omega};X)\\
x_0 + x_1  & \mapsto \biggl(x_0,   W_{p,X} \bigl( 
s \mapsto\chi(s) A^{\frac{1}{2}} T_s x_1\bigr)
\biggr). 
\end{align*}
Analogously, one can also define 
\begin{align*}
J_2 \co X^* =  \Ker (A^*) \oplus \ovl{\Ran(A^*)} & 
\longrightarrow  X^* \oplus_{p'} L^{p'}(\widehat{\Omega};X^*)\\
y_0 + y_1 & \mapsto \biggl(\frac{y_0}{2}, W_{p',X^*} \bigl(  
s \mapsto \chi(s) A^{*\frac{1}{2}}T_s^* y_1\bigr)\biggr). 
\end{align*}
For any $t\in \R$, we have
\begin{align*}
U_tJ_1(x_0 + x_1) = & \biggl(x_0,   V_t W_{p,X} \bigl( 
s \mapsto\chi(s) A^{\frac{1}{2}}T_s x_1\bigr)\biggr)\\
= & \biggl(x_0,   W_{p,X} \bigl( 
s \mapsto\chi(t+s) A^{\frac{1}{2}}T_{t+s} x_1\bigr)\biggr)
\end{align*}
by (\ref{shift}). Hence 
\begin{align*}
\bigl\langle U_t & J_1(x_0+ x_1) , J_2(y_0+y_1)\bigr\rangle
\\ & = 
\frac12\langle x_0, y_0\rangle\,+\,\Bigl\langle
W_{p,X} \bigl( 
s \mapsto\chi(t+s) A^{\frac{1}{2}}T_{t+s} x_1\bigr),
W_{p',X^*} \bigl( 
s \mapsto\chi(s) A^{*\frac{1}{2}}T_{s}^* y_1\bigr)\Bigr\rangle\\
& = 
\frac12\langle x_0, y_0\rangle\,+\,\int_0^\infty \bigl\langle 
A^{\frac{1}{2}}T_{t+s} x_1 , A^{*\frac{1}{2}}T_{s}^* y_1\bigr\rangle\, ds\\
& = 
\frac12\langle x_0, y_0\rangle\,+\,\int_0^\infty \bigl\langle 
AT_{2s} T_t x_1 ,  y_1\bigr\rangle\, ds,
\end{align*}
by Lemma \ref{duality}. For any $z\in \ovl{\Ran(A)}$, we have
$\int_0^\infty AT_{s}z = z$. Applying this identity to $z=T_t x_1$,
we deduce that
$$
\bigl\langle U_t J_1(x_0+ x_1) , J_2(y_0+y_1)\bigr\rangle
\,=\,\frac12\bigl(\langle x_0, y_0\rangle\,+\,\langle T_tx_1, y_1\rangle\bigr).
$$
Since $X$ is reflexive, we can apply (\ref{RNP}). Hence the above identity 
shows that $2J_2^*U_tJ_1 = T_t$, which concludes the proof.
\end{proof}

We conclude this section with an application to the problem of 
renorming $C_0$-semigroups. It is well-known that any bounded 
$C_0$-semigroup $(T_t)_{t\geq 0}$ on an arbitrary Banach space $(X, \norm{\cdot}_X)$  becomes 
contractive for the equivalent norm $\norm{x}' \coloneqq \sup_{t \ge 0} \norm{T_t x}_X$. 
However, this renorming may destroy regularity properties of the original norm.
For example by a classical result of Packel \cite{Packel} 
there exist bounded $C_0$-semigroups on Hilbert spaces which are not 
contractive for any equivalent Hilbert space norm (equivalently, which
are not similar to a contractive semigroup). The third author showed 
that among the bounded analytic semigroups on Hilbert space, 
those whose negative generator admits a bounded $H^{\infty}(\Sigma_\theta)$ functional 
calculus for some $\theta \in (0,\frac{\pi}{2})$ are exactly those which are
contractive for an equivalent Hilbert space norm, see \cite{Mer98} and 
\cite[Theorem 4.2]{LeMerdy-SurveySquareFunctions}. 
We prove a partial analogue of this result for uniformly convex renormings.

\begin{cor}\label{uc}
Let $(T_t)_{t \ge 0}$ be a bounded analytic $C_0$-semigroup on a 
uniformly convex Banach space $X$. Suppose that its negative
generator admits a bounded $H^{\infty}(\Sigma_{\theta})$ functional calculus 
for some $\theta \in (0, \frac{\pi}{2})$. Then there exists an equivalent uniformly 
convex norm on $X$ for which $(T_t)_{t \ge 0}$ is contractive.  
\end{cor}

\begin{proof}
We fix some $1<p<\infty$. 
The uniform convexity ensures that $X$ is reflexive and that $X$ and $X^*$
have finite cotype (see e.g. \cite[Thm. 1.e.16]{LinTza79}). 
Hence, by Theorem~\ref{Th dilation Delta semigroup}, 
there exist a measure space $\Omega'$, a $C_0$-group of isometries $(U_t)_{t \in\R}$ 
on the Bochner space $L^p(\Omega';X)$ together with two bounded 
operators $J \co X \to L^p(\Omega';X)$ and $Q \co L^p(\Omega';X) \to X$ such that
$$
T_t = QU_tJ, \qquad \text{for all } t \ge 0.
$$  
Then according to Proposition~\ref{compression}, there exist 
closed subspaces $F \subset E \subset L^p(\Omega';X)$ that are invariant under 
each operator $U_t$ and an isomorphism $S\co X \to E/F$ such that
$$
T_t= S^{-1} \widetilde{U}_t S,\qquad \text{for all } t \ge 0,
$$
where $(\widetilde{U}_t)_{t \ge 0}$ is the induced contractive 
semigroup on the quotient space $E/F$.

According to \cite{Figiel1, Figiel2} (see also \cite[Lemma 4.4]{Pisier-IHP}),
the space $L^p(\Omega';X)$ is uniformly convex. We deduce that
$E/F$ is uniformly convex as well. Now let
$$
\norm{x}' \coloneqq \norm{Sx}_{E/F},\qquad\hbox{for all}\  x \in X.
$$ 
Then $\norm{\cdot}'$ is a uniformly convex norm on $X$ for which $(T_t)_{t \ge 0}$ 
is contractive, and $\norm{\cdot}'$ is equivalent to the original norm.
\end{proof}

Using Theorem \ref{thm:positive_dilation} in the place of 
Theorem~\ref{Th dilation Delta semigroup}, we obtain the following 
analogue for Ritt operators.

\begin{cor}
Let $T$ be a Ritt operator on a 
uniformly convex Banach space $X$. Suppose that $T$ admits a bounded 
$H^{\infty}(B_{\gamma})$ functional calculus 
for some $\gamma \in (0, \frac{\pi}{2})$. 
Then there exists an equivalent uniformly 
convex norm on $X$ for which $T$ is contractive.  
\end{cor}

In the above two corollaries, uniform convexity could be replaced
by any Banach space property which is preserved by passing from 
$X$ to $L^p(X)$, and by passing to subspaces and quotients.
In particular this applies to the class of $SQ_p$-spaces. 
More results for this class will be given in Corollaries 
\ref{Th Ritt quotients of subspaces} and
\ref{Th SG quotients of subspaces}.


\section{Characterization Results}\label{sec:main_results}


In this section we consider Ritt operators or
bounded analytic semigroups and characterize bounded
$H^\infty$ functional calculus for them on various classes
of Banach spaces. 

We start with general UMD spaces. The idea expressed by the next statement
(and Theorem \ref{Th main UMD sectorial} below) is that  
any Ritt operator or bounded analytic semigroup with a bounded
$H^\infty$ functional calculus can be dilated into a contractive one with
a bounded $H^\infty$ functional calculus.
Next, more precise results will be given for operators
or semigroups acting on some specific classes of UMD spaces.

\begin{thm}
\label{thm:main_result}
Let $T \co X \to X$ be a Ritt operator on a UMD Banach space $X$ and let $1<p<\infty$. 
The following conditions are equivalent.
\begin{itemize}
\item [(i)] $T$ admits a bounded $H^{\infty}(B_\gamma)$ functional calculus for some 
$\gamma \in (0,\frac{\pi}{2})$. 
\item [(ii)] There exist a measure space $\Omega'$, a contractive 
Ritt operator $R \co L^p(\Omega';X) \to L^p(\Omega';X)$ which admits 
a bounded $H^{\infty}(B_{\gamma'})$ functional calculus for some $\gamma' \in (0,\frac{\pi}{2})$,
and two bounded operators $J \co X \to L^p(\Omega';X)$ and $Q \co L^p(\Omega';X)\to X$ such that
$$
T^n = QR^nJ, \qquad \text{for all } n \ge 0.
$$
\end{itemize}
If moreover $X$ is an ordered UMD Banach space, then the operator $R$ in {\rm (ii)} can be chosen to be positive.
\end{thm}

\begin{proof}
The implication `(ii)$\Rightarrow$(i)' is easy. Indeed if (ii) 
holds then we have $\phi(T) = Q\phi(R)J$ for
any $\phi\in \mathcal{P}$ and there is a constant 
$K\geq 0$ such that $\norm{\phi(R)}\leq 
K\norm{\phi}_{H^\infty(B_{\gamma'})}$ for any such $\phi$.
Consequently,
$$
\norm{\phi(T)}\leq K \norm{J}\norm{Q} \norm{\phi}_{H^\infty(B_{\gamma'})},
\qquad \text{for all }\phi\in\mathcal{P}.
$$
By Lemma \ref{pol}, this shows that $T$ has a bounded $H^{\infty}(B_{\gamma'})$ functional calculus.

\smallskip
Assume (i). Then it follows from Proposition~\ref{prop:fractional_hinfty} that for a
sufficiently small $\alpha > 1$, the fractional power $T_{\alpha}$ of $T$ has a 
bounded $H^{\infty}(B_{\gamma''})$ functional calculus for some $\gamma'' \in (0,\frac{\pi}{2})$. 
Since $X$ is UMD, it is reflexive and $X$ and $X^*$ have finite cotype. Hence we can apply
Theorem~\ref{thm:positive_dilation} to the operator $T_{\alpha}\co X\to X$. We obtain 
that there exist an isometric isomorphism $U \co L^p(\Omega';X) \to L^p(\Omega';X)$  
and bounded operators $J \co X \to L^p(\Omega';X)$ and 
$Q \co L^p(\Omega';X) \to X$ such that
\begin{equation}
\label{equa PT 0}
T_{\alpha}^n = QU^nJ, \qquad \text{for all } n \ge 0.
\end{equation}
This implies that for any polynomial $\phi \in \mathcal{P}$, we have
\begin{equation}
\label{equa PT}
\phi(T_{\alpha}) = Q\phi(U)J.	
\end{equation}
Let $\beta=\frac{1}{\alpha}$. Then we have $\beta\in (0,1)$.
By Theorem~\ref{thm:fractional_powers}, the fractional power 
$U_{\beta}$ is a contractive Ritt operator.
Moreover, by Theorem~\ref{thm:fractional_power_hinfty}, 
$U_{\beta}$ has a bounded $H^{\infty}(B_{\gamma'})$ 
functional calculus for some $\gamma' \in (0,\frac{\pi}{2})$.

Consider the polynomials $P_{m}(z) = \sum_{k=1}^m a_{\beta,k} z^k$, where the $a_{\beta,k}$'s are 
the coefficients in the series expansion of $1-(1-z)^{\beta}$ as given by (\ref{Dun1}).
Let $n\geq 0$ be an integer. Using equality (\ref{equa PT}) with the polynomial 
$\phi = P_{m}^n$, we see that
$$
\bigl(P_m(T_\alpha)\bigr)^n = P_{m}^n(T_\alpha)
=Q P_{m}^n(U)J = Q\bigl(P_m(U)\bigr)^nJ.
$$
Taking the limit when $m \to \infty$ on both sides yields
$$
T^n=Q U_\beta^n J,
$$
by (\ref{Dun2}). 
We deduce (ii) by setting $R=U_{\beta}$.

If $X$ is ordered, $U$ can be chosen to be positive by part (a) of 
Theorem~\ref{thm:positive_dilation}. Then $R$ is positive by part (c)
of Theorem~\ref{thm:fractional_powers}.

\end{proof}

We now focus on special classes of UMD spaces. We start with $L^p$-spaces.

\begin{thm}
\label{Th main Ritt Lp}
Let $\Omega$ be a measure space and $1<p<\infty$.
Let $T \co L^p(\Omega) \to L^p(\Omega)$ be a Ritt operator. 
The following conditions are equivalent.
\begin{itemize}
\item[(i)] $T$ admits a bounded $H^\infty(B_\gamma)$ functional 
calculus for some $\gamma \in (0, \frac{\pi}{2})$.
\item[(ii)] There exist a measure space $\Omega'$, a contractive 
and positive Ritt operator 
$R \co L^p(\Omega') \to L^p(\Omega')$ together with two bounded
operators $J \co L^p(\Omega) \to L^p(\Omega')$ and 
$Q \co L^p(\Omega') \to L^p(\Omega)$ such that
$$
T^n = Q R^n J \qquad \text{for all } n \ge 0.
$$
\end{itemize}
\end{thm}

\begin{proof}
By \cite[Theorem 3.3]{LeMXu12} (see also \cite[Theorem 8.3]{Mer12}), a positive contractive
Ritt operator on an $L^p$-space admits a bounded $H^\infty(B_\gamma)$ functional calculus 
for some $\gamma \in (0, \frac{\pi}{2})$. Moreover 
$L^p(\Omega';X)$ is an $L^p$-space whenever $X$ is an $L^p$-space. Hence
the result is a special case of Theorem \ref{thm:main_result}.
\end{proof}

If we take part (b) of Theorem~\ref{thm:positive_dilation} into account in the proof
of Theorem \ref{thm:main_result}, we obtain the following special case.

\begin{cor}
\label{Th Ritt subspaces}
Let $1<p<\infty$, let $X$ be a closed subspace of an $L^p$-space and
let $T$ be a Ritt operator on $X$. The following conditions are equivalent.
\begin{itemize}
\item[(i)] $T$ admits a bounded $H^\infty(B_\gamma)$ functional calculus for some $\gamma \in (0, \frac{\pi}{2})$.
\item[(ii)] There exist a measure space $\Omega'$, a contractive and positive Ritt operator 
$R \co L^p(\Omega') \to L^p(\Omega')$, an $R$-invariant subspace $E\subset L^p(\Omega')$,
together with two operators $J \co X \to E$ and 
$Q \co E \to X$ such that
$$
T^n = Q R^n J, \qquad \text{for all } n \ge 0.
$$
\end{itemize}
\end{cor}

For the class of quotients of subspaces of $L^p$, we have the following.

\begin{cor}
\label{Th Ritt quotients of subspaces}
Let $1<p<\infty$, let $X$ be an $SQ_p$-space and
let $T$ be a Ritt operator on $X$. The following conditions are equivalent.
\begin{itemize}
\item [(i)] $T$ admits a bounded $H^\infty(B_\gamma)$ functional 
calculus for some $\gamma \in (0, \frac{\pi}{2})$.
\item [(ii)] There exist a measure space $\Omega'$, 
a contractive and positive Ritt operator 
$R \co L^p(\Omega') \to L^p(\Omega')$, two $R$-invariant subspaces 
$F \subset E \subset L^p(\Omega')$ and an isomorphism $S \co X \to E/F$ such that
$$
T^n=S^{-1}\widetilde{R}^n S, \qquad \text{for all }  n \geq 0,
$$
where $\widetilde{R} \co E/F \to E/F$ is the compression of $R$ to $E/F$.
\end{itemize}
\end{cor}

\begin{proof}
To prove the implication `(ii) $\Rightarrow$ (i)', 
it suffices to use \cite[Theorem 3.3]{LeMXu12} 
as in the proof of Theorem \ref{Th main Ritt Lp}, together
with the fact that the boundedness of the functional calculus is preserved 
by passing to an invariant subspace, by factorizing through an invariant 
subspace and by similarity transforms. 

To prove the converse, assume (i) and apply the proof of 
Theorem \ref{thm:main_result} to this special case. Applying
part (c) of Theorem~\ref{thm:positive_dilation}, we can assume that
the operator $U\colon L^{p}(\Omega';X)\to L^p(\Omega';X)$ 
satisfying (\ref{equa PT 0}) is a compression of a positive 
isometric isomorphism $V\colon L^p(\Omega'')\to L^p(\Omega'')$.
By Proposition \ref{compression}, there exists 
a Banach space $Y$ which is a quotient of two $U$-invariant
subspaces of $L^{p}(\Omega';X)$, as well as
an isomorphism $S\colon X\to Y$ such that 
$$
T_\alpha^n = S^{-1} W^n S, \qquad \hbox{for all } n\geq 0,
$$
where $W\colon Y\to Y$ is the resulting compression of $U$. 
Then we may write $Y=E/F$ for some $V$-invariant 
subspaces $F\subset E\subset L^p(\Omega'')$, in such a way 
that $W$ is a compression of $V$. 

Let $q\colon E\to E/F$ be the canonical quotient map. For any polynomial 
$\phi$, we have 
$$
\phi(W)q = q\phi(V)_{\vert E}.
$$
Arguing as in the proof of Theorem \ref{thm:main_result} we deduce that
$$
W_\beta q = q V_{\beta\vert E},
$$
that is, $W_\beta$ is the compresion of $V_\beta$, and
$$
T^n = S^{-1} W_\beta^n S, \qquad \hbox{for all } n\geq 0.
$$
We deduce the result with $R=V_\beta$ and $\widetilde{R}=W_\beta$.
\end{proof}

\begin{remark}\label{other classes} 
If $X$ is a Banach lattice then $L^p(\Omega';X)$ also is a Banach lattice. 
So Theorem \ref{thm:main_result} shows that any Ritt
operator on a UMD Banach lattice with a bounded $H^\infty(B_\gamma)$ functional calculus  
can be dilated into a positive contractive Ritt operator acting on a bigger 
UMD Banach lattice, and admitting a bounded $H^\infty(B_{\gamma'})$ 
functional calculus for some $\gamma'<\frac{\pi}{2}$.

\smallskip
Likewise if $X$ is a noncommutative $L^p$-space with $1<p<\infty$, then 
$X$ is UMD \cite[Theorem 6.1]{BGM89} and $L^p(\Omega';X)$ is a noncommutative $L^p$-space. 
Consequently any Ritt operator on a noncommutative $L^p$-space with a bounded 
$H^\infty(B_\gamma)$ functional calculus  
can be dilated into a positive contractive Ritt operator acting on a bigger 
noncommutative $L^p$-space and admitting a bounded $H^\infty(B_{\gamma'})$ 
functional calculus for some $\gamma'<\frac{\pi}{2}$.
\end{remark}

We now turn to bounded analytic semigroups. The results below improve and
extend some of the main results by the second named author in~\cite{Fac13c}.

\begin{thm}
\label{Th main UMD sectorial}
Let $A$ be a sectorial operator on a UMD Banach space $X$ and let $1<p<\infty$. 
The following conditions are equivalent.
\begin{itemize}
\item [(i)] $A$ admits  a bounded $H^{\infty}(\Sigma_\theta)$ functional calculus 
for some $\theta \in (0,\frac{\pi}{2})$. 
\item [(ii)] There exist a measure space $\Omega'$, a sectorial 
operator $B$ of type $<\frac{\pi}{2}$ on $L^p(\Omega';X)$ which admits a bounded $H^{\infty}(\Sigma_{\theta'})$ 
functional calculus for some $\theta' \in (0,\frac{\pi}{2})$, 
and two bounded operators $J \co X \to L^p(\Omega';X)$ and 
$Q \co L^p(\Omega';X) \to X$ such that
$$
e^{-tA}=Qe^{-tB}J,\qquad \text{for all } t\geq 0,
$$
and
$$
\norm{e^{-tB}}\leq 1, \qquad \text{for all } t\geq 0.
$$
\end{itemize}
If moreover $X$ is an ordered UMD Banach space, then the sectorial operator $B$ in {\rm (ii)} can be chosen 
so that $e^{-tB}$ is positive for any $t\geq 0$.
\end{thm}

\begin{proof}
The proof of `(ii) $\Rightarrow$ (i)' is similar to the one for Theorem 
\ref{thm:main_result}, so we omit it.

Assume (i). According to Lemma \ref{fraction} we can find $\alpha > 1$ and
$\theta'' \in (0, \frac{\pi}{2})$ such that
the fractional power $A^{\alpha}$ has a bounded $H^{\infty}(\Sigma_{\theta''})$ 
functional calculus. Then it follows from Theorem~\ref{Th dilation Delta semigroup} 
applied to the operator $A^\alpha$ that there exist a measure space 
$\Omega'$, two bounded operators $J \co X \to L^p(\Omega';X)$ and 
$Q \co L^p(\Omega';X) \to X$ and a $C_0$-group $(U_t)_{t \in \R}$ of isometries 
(positive if $X$ is ordered) such that
\begin{align}  
e^{-tA^\alpha} = QU_tJ, \qquad \text{for all } t \ge 0. 
\label{eq:fundamental_factorization_semigroup}
\end{align}
We denote by $C$ the negative generator of $(U_t)_{t \geq 0}$, so that we can write 
$U_t=e^{-tC}$ for any $t\geq 0$. Since $X$ is UMD, the operator $C$ admits a bounded
$H^\infty(\Sigma_\omega)$ functional calculus for any $\omega>\frac{\pi}{2}$ \cite{HP}.
Let $\beta=\frac{1}{\alpha}$. Since this real number belongs to $(0,1)$, we deduce 
from Lemma \ref{fraction} that the operator $C^{\beta}$ is sectorial 
of type $<\frac{\pi}{2}$ and admits a bounded $H^{\infty}(\Sigma_{\theta'})$ 
functional calculus for some $\theta' \in (0,\frac{\pi}{2})$. 

We now use subordination. By \cite[IX,11]{Yos80}, for any $t> 0$, 
there exists a nonnegative function $f_{t,\beta} \in L^1(\R_+)$ with 
$\int_0^{\infty} f_{t,\beta}(s) \, ds = 1$, 
such that the semigroup $(e^{-tC^{\beta}})_{t \ge 0}$ 
generated by $-C^{\beta}$ is given in the strong sense by
$$
e^{-tC^{\beta}} = \int_0^{\infty} f_{t,\beta}(s) U_s \, ds.
$$
This explicit formula shows that $e^{-tC^{\beta}}$ 
is contractive for all $t \ge 0$ since each $U_s$ is contractive
(and positive if $X$ is ordered). Likewise we have
$$
e^{-tA}=
e^{-t(A^{\alpha})^{\beta}} =
\int_0^{\infty} f_{t,\beta}(s)e^{-sA^\alpha}\, ds
$$
for any $t>0$. These identities together 
with \eqref{eq:fundamental_factorization_semigroup} show 
that for any $t > 0$,
\begin{align*}
e^{-tA} = &
\int_0^{\infty} f_{t,\beta}(s) QU_sJ \, ds\\
& =Q\bigg(\int_0^{\infty} f_{t,\beta}(s) U_s \, ds\bigg)J\, =\ Qe^{-tC^{\beta}}J.
\end{align*}
We conclude by taking $B=C^{\beta}$.
\end{proof}

When we restrict to  $L^p$-spaces, we obtain the following result.

\begin{thm}
\label{Th main sectorial Lp}
Let $\Omega$ be a measure space and $1<p<\infty$. Let $A$ be a sectorial operator on $L^p(\Omega)$. 
The following conditions are equivalent.
\begin{itemize}
\item [(i)] $A$ admits a bounded $H^\infty(\Sigma_{\theta})$ functional calculus 
for some $\theta \in (0, \frac{\pi}{2})$.
\item [(ii)] There exist a measure space $\Omega'$, a sectorial operator
$B$ of type $<\frac{\pi}{2}$ on $L^p(\Omega')$, and two bounded operators
$J \co L^p(\Omega) \to L^p(\Omega')$ and $Q \co L^p(\Omega') \to L^p(\Omega)$ such that
$$
e^{-tA}=Qe^{-tB}J,\qquad \text{for all } t\geq 0
$$
and
$$
e^{-tB}\ \text{is a positive contraction}, \ \text{for all } t\geq 0.
$$
\end{itemize}
\end{thm}

\begin{proof}
(i) $\Rightarrow$ (ii) Since $L^p(\Omega';X)$ is an $L^p$-space whenever 
$X$ is an $L^p$-space, this implication is a special case of Theorem~\ref{Th main UMD sectorial}. 

(ii) $\Rightarrow$ (i) By \cite[Remark 4.c]{Weis01}, the operator $B$ admits a 
bounded $H^\infty(\Sigma_\theta)$ functional calculus for some $\theta \in (0, \frac{\pi}{2})$. 
Then the dilation assumption implies that $A$ has a bounded $H^{\infty}(\Sigma_\theta)$ functional calculus. 
\end{proof}

The next theorem is obtained by combining Theorem \ref{Th main UMD sectorial}
for subspaces of $L^p$ together with 
part (b) of Theorem~\ref{Th dilation Delta semigroup}.

\begin{cor}
\label{Th sectorial subspaces}
Let $1<p<\infty$, let $X$ be a closed subspace of an $L^p$-space
and let $A$ be a sectorial operator on $X$. 
The following conditions are equivalent.
\begin{itemize}
\item [(i)]  $A$ admits a bounded $H^\infty(\Sigma_{\theta})$ functional calculus for
some $\theta \in (0, \frac{\pi}{2})$.
\item [(ii)] There exist a measure space $\Omega'$, a sectorial operator
$B$ of type $<\frac{\pi}{2}$ on $L^p(\Omega')$, a subspace $E\subset L^p(\Omega')$ which is
$(e^{-tB})_{t\geq 0}$-invariant, and 
two bounded operators $J \co X \to E$ and $Q \co E \to X$ such that
$$
e^{-tA}=Qe^{-tB}J,\qquad \text{for all } t\geq 0,
$$
and
$$
e^{-tB}\co L^p(\Omega')\to L^p(\Omega') \ \text{is a positive contraction,}\ \text{for all } t\geq 0.
$$
\end{itemize}
\end{cor}

We finally consider semigroups acting on quotients of subspaces of $L^p$.
Arguing as in the proof of Corollary \ref{Th Ritt quotients of subspaces}, 
we recover the following result of the second named author
\cite{Fac13c}.

\begin{cor}\label{Th SG quotients of subspaces}
Let $1<p<\infty$, let $X$ be an $SQ_p$-space and let $A$ be a sectorial operator on $X$. 
The following conditions are equivalent.
\begin{itemize}
\item [(i)] $A$ admits a bounded $H^\infty(\Sigma_{\theta})$ functional calculus
for some $\theta \in (0, \frac{\pi}{2})$.
\item [(ii)] There exist a measure space $\Omega'$, a sectorial operator
$B$ of type $<\frac{\pi}{2}$ on $L^p(\Omega')$, two subspaces $F \subset E \subset L^p(\Omega')$ 
which are $(e^{-tB})_{t\geq 0}$-invariant, and an isomorphism $S \co X \to E/F$
such that
$$
e^{-tA}=S^{-1}\widetilde{e^{-tB}}S,\qquad \text{for all } t\geq 0,
$$
where $\widetilde{e^{-tB}}\co E/F\to E/F$ is the compression of $e^{-tB}$ to $E/F$, and
$$
e^{-tB}\co L^p(\Omega')\to L^p(\Omega') \ \text{is a positive contraction}, \  \text{for all }  t\geq 0.
$$
\end{itemize}
\end{cor}

Comments similar to the ones in Remark \ref{other classes} apply to the
sectorial setting.

\begin{remark} Let $\Omega'$ be a measure space, let $1<p<\infty$,
let $E\subset L^p(\Omega')$
be a closed subspace and let $(R_t)_{t\geq 0}$ be a bounded analytic semigroup
with generator $-B$. Assume that each $R_t\colon E\to E$
is contractively regular (in the sense of \cite{Pis94}). According
to \cite[Corollary 3.2]{LeMSim01},
$B$ admits a bounded $H^\infty(\Sigma_\theta)$ functional calculus for any 
$\theta>\frac{\pi}{2}$ (this can also be derived from the dilation result
\cite[Theorem 4.2.11]{Fac15}). However we do not know if $B$ admits  
a bounded $H^\infty(\Sigma_\theta)$ functional calculus for some 
$\theta\in(0, \frac{\pi}{2})$. If such a result were true,
it would be an analogue of Weis's Theorem 
\cite[Remark 4.c]{Weis01} for subspaces of $L^p$ and would 
lead to a more precise form of  Corollary \ref{Th sectorial subspaces}.

An essentially equivalent question is whether any contractively 
regular Ritt operator $R\colon E\to E$ admits a bounded 
$H^\infty(B_\gamma)$ functional calculus for some 
$\gamma\in(0, \frac{\pi}{2})$.
\end{remark}

\section{Representations of amenable groups}\label{sec:representations}

Let $G$ be a locally compact group, let $X$ be a Banach space and
let $\pi\colon G\to B(X)$ be a representation. We say that $\pi$
is continuous when $t\mapsto \pi(t)x$
is continuous for any $x\in X$. Then $\pi$ is said to be bounded 
when 
$$
\norm{\pi}: = \sup\bigl\{\norm{\pi(t)}\, :\, t\in G\bigr\}\, <\infty\,.
$$
A famous theorem of Dixmier asserts that if $G$ is amenable and
$X=H$ is a Hilbert space, then any bounded continuous representation
$\pi\colon G\to B(H)$ is similar to a unitary representation, that is, there
exists an isomorphism $S\in B(H)$ such that $S\pi(t)S^{-1}$ is a unitary 
for any $t\in G$ \cite{Dix}. 

The aim of this section is to 
establish Banach space analogues of that result.
In the Banach space context, the role of `unitary representations'
will be played by `isometric representations', that is, representations
$\pi$ such that $\pi(t)$ is an isometry for any $t$. Note that this 
holds true if and only if $\pi(t)$ is a contraction for any $t\in G$.

In the case when $G=\Z$ or $G=\R$, the results we obtain give an alternate
route to prove some of the results from Section \ref{sec:main_results},
see Remark \ref{alternative} for details.

\bigskip
We will use ultraproducts of Banach spaces. We pay a special attention 
to the case when
these spaces are ordered ones.
We recall that if $(X_j)_{j\in I}$ is a family of Banach spaces
indexed by an arbitrary set $I$ and $\U$ is an ultrafilter on $I$, then the 
ultraproduct $(X_j)^\U$ is defined as the quotient space $\ell^\infty(I;X_j)/N_\U$, where
$\ell^\infty(I;X_j)$ is the space of all bounded families $(x_j)_{j\in I}$ with $x_j\in X_j$, 
equipped with the sup norm, and $N_\U$ is the subspace of all such families for which
$\lim_\U\norm{x_j}_{X_j}=0$. Thus any element $z$ of $(X_j)^\U$ is a class of bounded 
families $(x_j)_{j\in I}$
modulo $N_\U$. Further if $(x_j)_{j\in I}$ is any representative of $z$, then
$$
\norm{z} = \,\lim_\U\norm{x_j}_{X_j}.
$$
We refer the reader to \cite{Khamsi-Sims-Ultra}
for general information on this construction.

When all spaces $X_j$ are equal to a single
space $X$, the associated ultraproduct is called an ultrapower 
and is denoted by $X^\U$.

Assume that each $X_j$ is the complexification of a real
Banach space $X_{j, \R}$. Let $(X_j)^\U_\R$ be the subset
of all elements of $(X_j)^\U$ which have a representative 
$(x_j)_{j\in I}$, with $x_j\in X_{j,\R}$ for any $j\in I$.
Clearly $(X_j)^\U_\R$ is a real subspace of $(X_j)^\U$.
Consider $z^1,z^2$ in $(X_j)^\U_\R$ and let 
$x_j^1, x^2_j \in X_{j,\R}$ such that
$(x^1_j)_{j\in I}$ and $(x^2_j)_{j\in I}$ are representatives
of $z^1$ and $z^2$, respectively.
Applying (\ref{2Complex}) to each $X_j$, and passing to the limit along $\U$,
we deduce that 
$\lim_{\U}\norm{x_j^1}_{X_j}\leq \lim_{\U}\norm{x_j^1 + i x^2_j}$, 
which means that
$$
\norm{z^1}\leq \norm{z^1+ iz^2}.
$$
This implies
that 
$(X_j)^\U_\R\cap i(X_j)^\U_\R=\{0\}$, and hence 
$(X_j)^\U$ is the real direct 
sum of $(X_j)^\U_\R$ and $i(X_j)^\U_\R$. Moreover 
$(X_j)^\U_\R$ is closed. Likewise, we have
$\norm{z^2}\leq \norm{z^1+ iz^2}$. Hence 
$(X_j)^\U$ is the complexification of $(X_j)^\U_\R$.

Assume now that each $X_j$ is a Riesz-normed space. 
We may define an order
on $(X_j)^\U_\R$ by asserting that $z\geq 0$ when it has a representative 
$(x_j)_{j\in I}$, with $x_j\geq 0$ for any $j\in I$. 
Then the corresponding positive cone
$\Cc= \{z\geq 0\}$ is closed (see e.g. 
the argument in \cite[p.224]{Heinrich}).

Let $z,z'\in (X_j)^\U_\R$, with $-z'\leq z\leq z'$.
Since $z'+z\geq 0$ and $z'-z\geq 0$ one can find, for each 
$j\in I$,
$x_j\geq 0$ and $y_j\geq 0$ in $X_j$ such that 
$(x_j)_{j\in I}$ and 
$(y_j)_{j\in I}$ are representatives of $z'+z$ and $z'-z$,
respectively. Then $\bigl(\frac12(y_j-x_j)\bigr)_{j\in I}$ and 
$\bigl(\frac12(y_j+x_j)\bigr)_{j\in I}$ are representatives of 
$z$ and $z'$, respectively. We have
$$
-(y_j+x_j)\leq
(y_j-x_j)\leq(y_j+x_j),\qquad 
j\in I.
$$
Since each $X_j$ is absolutely monotone, this implies
that $\norm{y_j-x_j}\leq \norm{y_j+x_j}$ for any $j\in I$.
Passing to the limit along $\U$, we deduce  that
$\norm{z}\leq \norm{z'}$. This shows that 
$\Cc$ is proper and that $(X_j)^\U$ is an absolutely
monotone ordered space.

Let $z\in (X_j)^\U_\R$, with representative $(x_j)_{j\in I}$ and let
$\epsilon>0$. Since each $X_j$ is Riesz-normed, one can find 
$y_j\in X_{j,\R}$ such that $-y_j\leq x_j\leq y_j$ and
$\norm{y_j}\leq(1+\epsilon)\norm{x_j}$, for any $j\in I$.
Let $w\in (X_j)^\U$ be the class of $(y_j)_{j\in I}$ (which
is a bounded family). Then $-w\leq z\leq w$
and passing to the limit, we have 
$\norm{w}\leq(1+\epsilon)\norm{z}$. This shows that 
$(X_j)^\U$ is a Riesz-normed space.

For further use we  note that for any $1<p<\infty$, the ultraproduct of a 
family of $L^p$-spaces is an $L^p$-space, see e.g.  
\cite[Example 2.20]{Khamsi-Sims-Ultra}.

Let $X$ be a Banach space. For any $1<p<\infty$
we denote by $\U(p;X)$ the class of Banach spaces which are ultraproducts 
of families of the form $(L^p(\Omega_j;X))_{j\in I}$ for some arbitrary measure 
spaces $\Omega_j$.

Let $G$ be a locally compact group, endowed with a fixed right Haar measure simply denoted by $dt$.
For a measurable set $E\subset G$, we
let $\vert E\vert$ denote its Haar measure. Then we have
$\vert Es\vert =\vert E\vert$ for any $s\in G$. 
Let $E \triangle F$ denote the symmetric difference of two subsets of $G$.
A net $(E_i)_{i\in I}$ of measurable subsets of $G$ is called a \emph{F\o lner net} if 
$0< \vert E_i \vert<\infty$ for any $i\in I$ and 
\begin{equation}\label{Folner1}
\lim_i\, \frac{\vert E_i s\triangle E_i\vert}{\vert E_i\vert} = 0, 
\qquad \text{for all } s \in G.
\end{equation}
The existence of a  F\o lner net is equivalent to the 
amenability of $G$ (see \cite{Pat} for details and other 
equivalent definitions).

Let $E\subset G$ be a measurable set such that $0< \vert E \vert <\infty$ and let $s\in G$. 
Consider $F_1 = E\setminus (Es \cap E)$ and $F_2 = Es\setminus (Es \cap E)$.
Since $\vert Es\vert =\vert E\vert$, we have $\vert F_1\vert = \vert F_2\vert$. 
Moreover $Es \triangle E$ is the disjoint union of $F_1$ and $F_2$, hence 
$\vert Es \triangle E\vert =2\vert F_1\vert$. 
Since $E$ is the disjoint union of $Es\cap E$ and $F_1$, we also
have  $\vert E\vert = \vert Es \cap E \vert + \vert F_1\vert$. 
Altogether, we obtain that
$$
\vert Es \cap E \vert +\tfrac{1}{2}\vert Es\triangle E\vert = \vert E\vert.
$$
Consequently
$$
\frac{\vert Es\cap E\vert}{\vert E\vert}\,= 1-\frac{1}{2}
\frac{\vert Es\triangle E\vert}{\vert E\vert}.
$$
Thus if $(E_i)_{i \in I}$ is a  F\o lner net on $G$, then
\begin{equation}\label{Folner2}
\lim_i\, \frac{\vert E_i s\cap E_i\vert}{\vert E_i\vert} = 1, 
\qquad \text{for all } s \in G.
\end{equation}

In the sequel we let 
$$
\kappa_X \co X\hookrightarrow X^{**}
$$ 
denote the canonical embedding of $X$ into its bidual.

\begin{thm}
\label{thm:factorization_amenable_groups} 
Let $\pi \co G \to \mathcal{B}(X)$ be a bounded 
continuous representation of an amenable locally compact 
group $G$ on a Banach space $X$. Suppose $1<p<\infty$. 
\begin{itemize}
\item [(1)] There exist a Banach space $Y$ in the class $\U(p;X)$, 
an isometric representation $\widehat{\pi} \co G \to\mathcal{B} (Y)$ and two 
bounded operators $J \co X \to Y$ and $Q \co Y \to X^{**} $ 
such that $\norm{J}\norm{Q} \leq \norm{\pi}^2$ and 
$$
\kappa_X\pi(t)=Q\widehat{\pi}(t)J, 
\qquad \text{for all } t \in G.
$$
\item [(2)] For any $x \in X$, the map $t \mapsto \widehat{\pi}(t)J(x)$ from 
$G$ into $Y$ is continuous.
\end{itemize}
Moreover, if $X$ is a Riesz-normed space, then $Y$ is a Riesz-normed space as well and 
the representation $\widehat{\pi}$ can be chosen such that $\widehat{\pi}(t)$ is 
a positive operator on $Y$ for any $t \in G$.
\end{thm}

\begin{proof}
Let $E\subset G$ be a measurable set with $0<\vert E\vert<\infty$. We let
$j_E \co X\to L^p(G;X)$ be the linear map given by
$$
j_{E}(x) = \frac{1}{\vert E\vert^{\frac{1}{p}}}\chi_E\,\pi(\cdot)x
$$
for any $x \in X$. This is well-defined; indeed,
the continuity of $\pi$ shows that $j_{E}(x)$ is measurable and we have
$$
\norm{j_{E}(x)}_{L^p(G;X)}^p = \,
\frac{1}{\vert E\vert}\int_E\norm{\pi(t)x}_X^p\, dt\,\leq \norm{\pi}^p\norm{x}_X^p.
$$
This shows that $j_{E}$ is bounded with $\norm{j_E} \leq \norm{\pi}$. Let $p'$ be the conjugate 
of $p$. Then similarly  we define $q_{E} \co L^p(G;X)\to X$ by
$$
q_{E}(f) = \frac{1}{\vert E\vert^{\frac{1}{p'}}}\,\int_E\pi(t^{-1})\bigl(f(t)\bigr)\, dt
$$
for any $f \in L^p(G;X)$. Using H\"older's inequality we see that
\begin{equation}
\begin{split}
\bnorm{q_E(f)}_X & \leq \frac{1}{\vert E\vert^{\frac{1}{p'}}} \int_E 
\bnorm{\pi(t^{-1})\bigl(f(t)\bigr)}_X\, dt \leq \frac{\norm{\pi}}{\vert E\vert^{\frac{1}{p'}}} \int_E\bnorm{f(t)}_X\, dt \\
&\leq \norm{\pi} \Bigl(\int_E\bnorm{f(t)}_X^p\, dt \Bigr)^{\frac{1}{p}} = \norm{\pi} \norm{f}_{L^p(G;X)}.
\end{split}
\end{equation}
This shows that $\norm{q_E} \leq \norm{\pi}$.

For any $s \in G$ let $\tau_s \co L^p(G;X) \to L^p(G;X)$ be the isometry 
given by the right regular vector-valued representation:
$\bigl(\tau_s(f)\bigr)(t) = f(ts)$ for any $f \in L^p(G;X)$ and any $t \in G$. 
Let $x \in X$ and $s \in G$. For $E$ as above, we have 
\begin{equation}
\bigl(\tau_s j_{E}(x)\bigr)(t) =\frac{1}{\vert E\vert^{\frac{1}{p}}} \chi_E(ts) \pi(ts)x
= \frac{1}{\vert E\vert^{\frac{1}{p}}} \chi_{Es^{-1}}(t) \pi(t)\pi(s)x 
\label{eq:inclusion_and_shift}
\end{equation}
for all $t\in G$. 
Applying $q_E$ to both sides of the above equality we obtain
\begin{equation}
q_{E} \tau_s j_{E}(x) = \frac{1}{\vert E\vert} \int_E 
\chi_{Es^{-1}}(t) \pi(s)x\, dt \, =\,\frac{\vert Es^{-1}\cap E\vert}{\vert E\vert}\, \pi(s)x. 
\label{Compo}
\end{equation}

Now let $(E_i)_{i\in I}$ be a F\o lner net on $G$ and form, for each $i\in I$, the operators
$j_{E_i}$ and $q_{E_i}$ as above.
Let $\U$ be an ultrafilter on $I$ refining the filter generated by the order 
of $I$. Then let $Y=L^p(G;X)^\ul$ be the ultrapower of $L^p(G;X)$ with respect to $\U$. 
For any $x\in X$, the family $(j_{E_i}(x))_i$ is bounded and we may therefore
define $J \co X \to Y$  by  $J(x) = (j_{E_i}(x))_i^{\bullet}$, the
class of the family $(j_{E_i}(x))_i$ in the ultrapower $Y$. Then we have
$\norm{J} \leq \norm{\pi}$. 
 
Let $(f_i)_{i\in I}$ belong to $\ell^\infty(I; L^p(G;X))$ and 
let $K=\sup_i\norm{f_i}$. We have
$$
\bnorm{\kappa_X q_{E_i}(f_i)}_{X^{**}}\,=\,\norm{q_{E_i}(f_i)}_{X} \leq K \norm{\pi}.
$$
Since bounded sets of $X^{**}$ are $w^*$-compact, we deduce the existence of the weak$^*$ limit 
$w^*-\lim_{\ul} \kappa_X q_{E_i}(f_i)$ in $X^{**}$. Furthermore this limit does only depend on the class
of $(f_i)_{i\in I}$ in $Y$.
Then we may define a map
\begin{equation*}
\begin{array}{cccc}
 Q  \co &         Y        &  \longrightarrow &  X^{**}    \\
        &  (f_i)^{\bullet} &  \longmapsto     &  w^*-\lim_{\ul} \kappa_X q_{E_i}(f_i),\\
\end{array}
\end{equation*}
this map is linear and by the above estimates,
we have $\norm{Q} \leq \norm{\pi}$.

Next for any $s\in G$ we denote by $\widehat{\pi}(s) \co Y \to Y$ the map induced by 
the operators $\tau_s$. That is, for any $(f_i)_{i\in I}$ in $\ell^\infty(I; L^p(G;X))$,
$$
\widehat{\pi}(s)\bigl((f_i)_i^\bullet\bigr)\, =\, 
\bigl((\tau_s (f_i))_i^\bullet\bigr).
$$
It is clear that $\widehat{\pi}\colon G \to \mathcal{B}(Y)$ is an isometric 
(a priori non-continuous) representation. 

If $X$ is Riesz-normed, then $L^p(G;X)$ is Riesz-normed and then the ultrapower
$Y$ is a Riesz-normed space, as explained 
before the statement of Theorem \ref{thm:factorization_amenable_groups}. In this case,
$\tau_s$ is positive and $\widehat{\pi}(s)$ is positive for any $s\in G$. 

Let $x\in X$, $\eta\in X^*$ and $s\in G$. Recall that $\U$ refines the order of $I$.
Then by \eqref{Compo} and the F\o lner condition 
(\ref{Folner2}), we have
\begin{align*}
\big\langle \eta, Q\widehat{\pi}(s)Jx\big\rangle_{X^*,X^{**}} & = \lim_{\U} \big\langle \eta,
q_{E_i}\tau_sj_{E_i}(x)\big\rangle_{X^*,X} \\
& = \lim_\U \frac{\vert E_is^{-1}\cap E_i\vert}{\vert E_i\vert}\, \big\langle \eta, \pi(s)x\big\rangle_{X^*,X}\\
& = \big\langle \eta, \pi(s)x\big\rangle_{X^*,X}.
\end{align*}
This shows the factorization property in part (1).

Let us now prove part (2). Let $x \in X$. 
We fix $s \in G$.
Consider as before a measurable set $E \subset G$ with $0<\vert E\vert <\infty$.
We have seen in \eqref{eq:inclusion_and_shift} that 
$$
\tau_s j_{E}(x) = \frac{1}{\vert E\vert^{\frac{1}{p}}} \chi_{Es^{-1}} \pi(\cdot)\pi(s)x.
$$
Hence,
\begin{align*}
\tau_s j_{E}(x) - j_{E}(x) & = \chi_{Es^{-1}} \frac{\pi(\cdot)\pi(s)x}{\vert E\vert^{\frac{1}{p}}}
- \chi_{E} \frac{\pi(\cdot)x}{\vert E\vert^{\frac{1}{p}}} \\
& = \frac{\bigl(\chi_{Es^{-1}} -\chi_{E}\bigr)\pi(\cdot)x}{\vert E\vert^{\frac{1}{p}}} 
+ \frac{\chi_{Es^{-1}} \pi(\cdot)\bigl(\pi(s)x -x\bigr)}{\vert E\vert^{\frac{1}{p}}}.
\end{align*}
We estimate the norm of each term in $L^p(G;X)$. On the one hand we have
\begin{align*}
\Bgnorm{\frac{\bigl(\chi_{Es^{-1}} -\chi_{E}\bigr)\pi(\cdot)x}{\vert E\vert^{\frac{1}{p}}}}_{L^p(G;X)}^p
& = \frac{1}{\vert E\vert} \int_G \bigl\vert \chi_{Es^{-1}}(t) -\chi_{E}(t)\bigr\vert
\norm{\pi(t)x}_X^p\, dt \\
& \leq \norm{\pi}^p \norm{x}_X^p \frac{\vert Es^{-1}\triangle E\vert}{\vert E\vert}.
\end{align*}
On the other hand one has
\begin{align*}
\Bgnorm{\frac{\chi_{Es^{-1}} \pi(\cdotp)\bigl(\pi(s)x -x\bigr)}{\vert E\vert^{\frac{1}{p}}}}_{L^p(G;X)}^p
& = \frac{1}{\vert E\vert} \int_{Es^{-1}} \bnorm{\pi(t)\bigl(\pi(s)x -x\bigr)}_X^p\, dt\\
& \leq \norm{\pi}^p\norm{\pi(s)x -x}_X^p.
\end{align*}
We deduce that
$$
\bnorm{\tau_s j_{E}(x) - j_{E}(x)}_{L^p(G;X)} 
\leq \norm{\pi} \biggl(\frac{\vert Es^{-1}\triangle 
E\vert^{\frac{1}{p}}}{\vert E\vert^{\frac{1}{p}}}\,\norm{x} 
+ \norm{\pi(s)x -x}\biggr).
$$
Applying this inequality to the sets $E_i$ of the  F\o lner 
net considered in the proof of (1) and using (\ref{Folner1}), we obtain
$$
\bnorm{\widehat{\pi}(s)Jx -Jx}_Y = \lim_{\U}\bnorm{\tau_s j_{E_i}(x) - j_{E_i}(x)}_{L^p(G;X)} 
\leq \norm{\pi} \norm{\pi(s)x-x}_X.
$$
This estimate and the continuity of $\pi(\cdot)x$ show that $\widehat{\pi}(\cdot)Jx$
is continuous at the origin, and hence on $G$.
\end{proof}

In general, the ultrapower $Y$ introduced in the above
proof is too big to expect
the representation $\widehat{\pi}$ to be continuous. 
This defect can be avoided by passing
to a suitable subspace. More precisely we have 
the following corollary 
(relevant only in the case when $G$ is not a 
discrete group). Its proof 
relies on notions and results from the paper \cite{deLGli65}
concerning possibly discontinuous representations.

\begin{cor}
\label{cor:factorization_continuous_representation}
Let $\pi \co G \to \mathcal{B}(X)$ be a bounded 
continuous representation of an amenable locally compact 
group $G$ on a Banach space $X$. 
Suppose $1<p<\infty$. 
\begin{itemize}
\item [(1)] There exist a  Banach space $Y$ in the class $\U(p;X)$, a subspace 
$Z\subset Y$, a continuous isometric representation 
$\widehat{\pi} \co G \to\mathcal{B} (Z)$ and two 
bounded operators $J \co X \to Z$ and $Q \co Z \to X^{**}$ 
such that $\norm{J}\norm{Q} \leq \norm{\pi}^2$ and 
$$
\kappa_X \pi(t)=Q\widehat{\pi}(t)J, 
\qquad \text{for all } t \in G.
$$
\item [(2)] Assume further that $X$ and $X^*$ are uniformly convex.
\begin{itemize}
\item [(2.i)] 
Then $Z$ can be chosen to be 1-complemented in $Y$, i.e. there exists
a contractive projection $P\colon Y\to Y$ with range equal to $Z$. 
\item [(2.ii)] 
Moreover if $X$ is a Riesz-normed space, then $Y$ is a Riesz-normed space, $Z$ is a Riesz-normed subspace of $Y$,
$\widehat{\pi}(t)\colon Z\to Z$ is positive for any
$t\in G$, and the contractive projection $P$ is positive.
\end{itemize}
\end{itemize}
\end{cor}

\begin{proof} 
We start from the space $Y$ and the representation $\widehat{\pi}$
constructed in Theorem \ref{thm:factorization_amenable_groups}. Then we set
$$
Z = \,\bigl\{y \in Y \,:\, s\mapsto \widehat{\pi}(s)y 
\text{ is continuous from } G \text{ into } Y\bigr\}.
$$
It is plain that $Z$ is a closed subspace of $Y$. For any 
$s_0\in G$ and any $y\in Z$, we have $\widehat{\pi}(s)(\widehat{\pi}(s_0)y) = \widehat{\pi}(ss_0)y$
for any $s\in G$. Hence $s\mapsto  \widehat{\pi}(s)(\widehat{\pi(s_0)}y)$ is continuous.
This shows that $Z$ is $\pi$-invariant.
We keep the notation $\widehat{\pi}\colon G\to \mathcal{B}(Z)$ for the 
representation obtained by taking restrictions. 
By construction, this representation is continuous. By part (2) of
Theorem \ref{thm:factorization_amenable_groups}, $Z$ contains the range of $J$.
Then changing $Q$ into its restriction to $Z$, we obtain part (1) of the corollary.

Assume now that $X$ and $X^*$ are uniformly convex. 
As already mentioned in the proof of Corollary \ref{uc},
the Bochner spaces
$L^p(G;X)$ and $L^{p'}(G;X^*)$ are uniformly convex. 
By \cite[Example 2.17]{Khamsi-Sims-Ultra}, the ultrapower $Y=L^p(G;X)^{\U}$ 
is uniformly convex as well. Hence by \cite[Theorem 2.19]{Khamsi-Sims-Ultra}, 
we can isometrically identify the dual of $Y$  
with the ultrapower $L^{p'}(G;X^*)^{\U}$. Applying 
\cite[Example 2.17]{Khamsi-Sims-Ultra} again,  this
dual space $Y^*$ is also uniformly convex.

Following \cite{deLGli65}, let $\mathcal{V}$ be the set of all 
neighbourhoods $V$ of the identity of $G$. Then for any 
such $V$, consider $\overline{\{\widehat\pi(t): t \in V \}}^{wo}$,
the closure being taken in the weak operator topology of $\mathcal{B}(Y)$. Then
let $\mathcal{S}$ be the closure in the weak operator topology of 
the convex hull of $\bigcap_{V \in \mathcal{V}} 
\overline{\{\widehat\pi(t): t \in V \}}^{wo}$. According 
to \cite[Theorem 3.1]{deLGli65} and its proof, the
uniform convexity of $Y$ and of $Y^*$ ensure that 
$\mathcal{S}$ contains a projection $P\colon Y\to Y$
with range equal to $Z$.

Since $\widehat\pi(t)$ is a contraction for any $t\in G$, any element
of $\mathcal{S}$ is a contraction. Consequently, $P$ is a contractive projection.

Assume now that $X$ is a Riesz-normed space. We already 
noticed that the ultrapower $Y$ is a Riesz-normed space and according to
Theorem \ref{thm:factorization_amenable_groups}, $\widehat{\pi}(t)\colon Y\to Y$ is positive 
for any $t\in G$. We deduce that $\langle P(y), y^*\rangle\geq 0$
for any $y\in Y$ and $y^*\in Y^*$. By Hahn-Banach,
this implies that $P$ is positive. (In fact, any element of 
$\mathcal{S}$ is positive.) Since $Y$ is Riesz-normed, this
implies that $P$ is real.

Let $y\in Z$ and let $y_1,y_2\in Y_{\R}$ such that
$y=y_1+iy_2$. Since $Z$ is the range of $P$, we have
$y=P(y_1)+i P(y_2)$. Since $P$ is a real operator, 
$P(y_1)$ and $P(y_2)$ belong to $Y_\R$. Consequently,
$P(y_1)=y_1$ and $P(y_2)=y_2$. Thus $y_1$ and
$y_2$ belong to $Z$. Define $Z_\R = Z\cap Y_\R$.
The above reasoning shows that 
$Z=Z_R\oplus i Z_\R$ and we immediatly deduce that
$Z$ is an ordered subspace of $Y$. It inherits the absolute
monotonicity of $Y$.

Finally consider $y\in Z$ and $\epsilon>0$. Since 
$Y$ is Riesz-normed, 
there exists $y'\in Y$ such that $-y'\leq y\leq y'$
and $\norm{y'}\leq (1+\epsilon)\norm{y}.$ Since $P(y)=y$ and
$P$ is a positive contraction, we both have
$-P(y')\leq y\leq P(y')$
and $\norm{P(y')}\leq (1+\epsilon)\norm{y}.$
Since $P(y')\in Z$, this shows that 
$Z$ is a Riesz-normed space.
\end{proof}

We now consider the special case when $X$ is an $L^p$-space for $1<p<\infty$. 
Then any element in $\U(p,X)$ is an $L^p$-space as well. Moreover the range of a positive
contractive projection on an $L^p$-space is positively and isometrically isomorphic 
to an $L^p$-space (see e.g. \cite[Theorem~4.10]{Ran01}). We therefore deduce the
following corollary.

\begin{cor}\label{DDLp}
Let $1<p<\infty$, let $\Omega$ be a measure space  and 
let $\pi \co G \to \mathcal{B}(L^p(\Omega))$ be a bounded continuous 
representation of an amenable locally compact group $G$. Then there exist a measure space $\Omega'$, a 
continuous isometric representation $\pi' \co G \to \mathcal{B}(L^p(\Omega'))$ such 
that $\pi'(t)$ is positive for any $t\in G$, and two bounded 
operators $J \co L^p(\Omega) \to L^p(\Omega')$ and 
$Q \co L^p(\Omega') \to L^p(\Omega)$ such that 
$\norm{J}\norm{Q} \leq \norm{\pi}^2$ and 
$$
\pi(t) = Q\pi'(t)J, \qquad \text{for all } t \in G.
$$
\end{cor}

\begin{remark}\label{alternative} (1)$\,$
In the case when $G=\Z$, the above corollary means the following: whenever
$T\colon L^p(\Omega)\to L^p(\Omega)$ is an invertible operator such that
$\sup_{n\in\Z}\norm{T^n}\,<\infty$, then there exist a measure space $\Omega'$,
an isometric isomorphism $U\colon L^p(\Omega')\to L^p(\Omega')$
and bounded operators 
$J \co L^p(\Omega) \to L^p(\Omega')$ and 
$Q \co L^p(\Omega') \to L^p(\Omega)$ such that $T^n = QU^nJ$
for any $n\in \Z$. This result was shown to the first and third author 
by \'Eric Ricard in 2011 as a way to improve \cite[Theorem 4.8]{ArhLeM11}. 
The argument in the proof of the first part of Theorem 
\ref{thm:factorization_amenable_groups} is an extension 
of Ricard's original argument.

\smallskip
(2)$\,$
In the case when $X$ is an $L^p$-space, Theorem 
\ref{thm:positive_dilation} can be obtained by combining 
the above result (Corollary \ref{DDLp} for $G=\Z$), \cite[Theorem 6.4]{Mer12} and
\cite[Theorem 4.8]{ArhLeM11}. 

Likewise, Theorem \ref{Th dilation Delta semigroup} in the case 
when $X$ is an $L^p$-space can be obtained by combining 
Corollary \ref{DDLp} for $G=\R$ and \cite[Section 5]{FroWei06}. Details are left 
to the reader.

\end{remark}


We now derive analogues of Dixmier's Theorem in this 
Banach space setting.

\begin{thm} 
\label{Th of similarity}
Let $G$ be an amenable locally compact group, let $X$ be a reflexive Banach space and let
$\pi \co G \to \mathcal{B}(X)$ be a bounded continuous 
representation of $G$ on $X$. Suppose $1<p<\infty$. 
Then there exist a Banach space $\widetilde{X}$ which is a 
quotient of a subspace of an element of $\U(p;X)$,
and an isomorphism $S \co X \to \widetilde{X}$ such that
$\norm{S}\norm{S^{-1}} \le \norm{\pi}^2$ and 
\begin{equation*}
\begin{array}{cccc}
       &   G  &  \longrightarrow   &  \mathcal{B}(\widetilde{X})  \\
           &  t   &  \longmapsto       &  S\pi(t)S^{-1}  \\
\end{array}
\end{equation*}
is an isometric representation of $G$ on $\widetilde{X}$.
\end{thm}

\begin{proof}
By Theorem~\ref{thm:factorization_amenable_groups} there exist a space $Y$ in $\U(p;X)$, an isometric 
representation $\widehat{\pi} \co G \to \mathcal{B}(Y)$ and two bounded operators 
$J \co X \to Y$ and $Q \co Y \to X$ such that $\norm{J}\norm{Q} \leq \norm{\pi}^2$ and 
$$
\pi(t)=Q\widehat{\pi}(t)J, \qquad \text{for all } t \in G. 
$$
According to Proposition~\ref{compression}, there exist two 
$\widehat{\pi}$-invariant closed subspaces $F \subset E \subset Y$ and an isomorphism 
$S \co X \to E/F$ with $\norm{S}\norm{S^{-1}} \le \norm{J} \norm{Q}$ 
such that the compressed representation 
$\widetilde{\widehat{\pi}}(t) \co E/F \to E/F$ of $\widehat{\pi}(t)$ satisfies
$$
\pi(t) = S^{-1} \widetilde{\widehat{\pi}}(t) S, \qquad \text{for all } t \in G. 
$$
Each $\widetilde{\widehat{\pi}}(t)$ is a contraction hence we obtain the result
with $\widetilde{X}=E/F$.
\end{proof}

Note that in the above situation, the representation $t\mapsto S\pi(t)S^{-1}$
is necessarily continuous, although $\widehat{\pi}$ may be discontinuous.

Specializing to $L^p$-spaces, we obtain the following corollary.

\begin{cor} 
\label{Cor of similarity}
Let $G$ be an amenable locally compact group, let $\Omega$ be a measure space, let
$1<p<\infty$ and let 
$\pi \co G \to \mathcal{B}(L^p(\Omega))$ be a bounded continuous representation.  
Then there exist an $SQ_p$-space $\widetilde{X}$
and an isomorphism $S \co X \to \widetilde{X}$ such that 
$\norm{S}\norm{S^{-1}} \le \norm{\pi}^2$ and
\begin{equation*}
\begin{array}{cccc}
       &   G  &  \longrightarrow   &  \mathcal{B}(\widetilde{X})  \\
           &  t   &  \longmapsto       &  S\pi(t)S^{-1}  \\
\end{array}
\end{equation*}
is an isometric representation of $G$ on $\widetilde{X}$.
\end{cor} 

Note that the class of $SQ_2$-spaces coincide with the class of
Hilbert spaces. Hence Dixmier's Theorem corresponds to
the case $p=2$ in the above corollary.

Except when $p=2$ and $X$ is a Hilbert space, Theorem  
\ref{thm:factorization_amenable_groups} is a stronger (more precise) result than 
Theorem \ref{Th of similarity}.

We conclude this paper with an application to unconditional bases (more generally to
unconditional decompositions). Consider a 
Schauder decomposition $(X_k)_{k\geq 1}$ of a Banach space X, and let
$(Q_k)_{k \geq 1}$ be the sequence of associated projections onto $X_k$. Thus for any
$j\not=k$, we have $X_j\subset {\rm Ker}(Q_k)$ 
and ${\rm Ran}(Q_k)=X_k$. Moreover 
$$
x =\sum_{k=1}^{\infty} Q_k x,\qquad \hbox{ for all } x \in X.
$$
The Schauder decomposition is called
unconditional if the above series 
is unconditionally convergent for any $x \in X$.
Recall that this holds true if and only if there exists a positive 
constant $C$ such that
\begin{equation}\label{uncond}
\bgnorm{\sum_{k=1}^{n}\omega_k Q_kx}_{X} \leq C\bgnorm{\sum_{k=1}^{n} Q_kx}_{X}
\end{equation}
for any choices of $\omega_k \in\{-1,1\}$, $x \in X$ and $n \in \mathbb{N}$. 
We necessarily have $C\geq 1$ and the smallest $C \geq 1$ satisfying (\ref{uncond}) 
is called the unconditional constant of the decomposition.

\begin{thm}
Let $(X_k)_{k\geq 1}$ be an unconditional decomposition of a 
reflexive Banach space $X$ with unconditional constant $C\geq 1$. 
Suppose $1<p<\infty$. Then there exist a Banach space $\widetilde{X}$ which is  
a quotient of a subspace of an element of $\U(p;X)$ and an isomorphism 
$S \co X \to \widetilde{X}$ with $\norm{S}\norm{S^{-1}} \leq C^2$ 
such that the unconditional constant of the unconditional decomposition 
$(S(X_k))_{k \geq 1}$ of  $\widetilde{X}$ is equal to $1$.
\end{thm}

\begin{proof}
For any $\omega \in \{-1,1\}^\mathbb{N}$, we can consider the bounded linear operator
\begin{equation*}
\begin{array}{cccc}
   \pi(\omega)   \co &  X   &  \longrightarrow   &  X  \\
           &  x=\sum_{k} Q_k(x)   &  \longmapsto       &  \sum_{k}\omega_k Q_k(x).  \\
\end{array}
\end{equation*}
Then, we obtain a bounded continuous representation 
$\pi \co \{-1,1\}^\mathbb{N} \to \mathcal{B}(X)$ of 
the compact group $\{-1,1\}^\mathbb{N}$, with $\norm{\pi} \leq C$. 
By Theorem \ref{Th of similarity}, there exist a Banach space 
$\widetilde{X}$ which is a quotient of a subspace of an element 
of $\U(p;X)$ and an isomorphism $S \co X \to \widetilde{X}$ with 
$\norm{S}\norm{S^{-1}} \le C^2$ such that
	\begin{equation*}
\begin{array}{cccc}
       &   \{-1,1\}^\mathbb{N}  &  \longrightarrow   &  \mathcal{B}(\widetilde{X})  \\
           &   \omega  &  \longmapsto       &  S\pi(\omega)S^{-1}  \\
\end{array}
\end{equation*}
is an isometric representation of the group $\{-1,1\}^\mathbb{N}$ 
on the Banach space $\widetilde{X}$. The sequence $(S(X_k))_{k \geq 1}$
is an unconditional decomposition of $\widetilde{X}$ whose associated 
projections are equal to $SQ_kS^{-1}$. Then it is easy to check that 
its unconditional constant is equal to $1$.
\end{proof}

If $X=H$ is a Hilbert space then we recover the classical result (see e.g.
\cite[Theorem 3.1.4]{Nikolski-Book2}) which says that if $(H_k)_{k \geq 1}$ 
is an unconditional decomposition of $H$, 
then there exists an isomorphism $S \co H \to H$ such that 
$(S(H_k))_{k \geq 1}$ is an orthogonal decomposition.

\bigskip
\textbf{Acknowledgement}. The authors would like to thank
the organizers of the conference ``GDR 2013 Analyse Fonctionnelle, 
 Harmonique et Probabilit\'es'', where
this project started, and \'Eric Ricard for his decisive contribution 
to Section 6. They also thank  Alexandre Nou and Uwe Franz for some fruiful discussions.
The first and third named authors are supported by the research program ANR 2011 BS01 008 01.
The second named author is supported by a scholarship of the
``Landesgraduiertenf\"orderung Baden-W\"urttemberg".
They finally thank the anonymous referee for his valuable suggestions 
which improved the presentation of the paper.


\bibliographystyle{amsalpha}
\bibliography{Ref2-AFL}{}

\end{document}